\documentclass[preprint,11pt]{elsarticle}
\usepackage{amsfonts}
\usepackage{amsmath}
\usepackage{amssymb}
\usepackage{latexsym}
\usepackage{graphicx}
\usepackage{epstopdf}
\usepackage{ntheorem}
\usepackage{arydshln}
\usepackage[usenames]{color}

\newtheorem{theo}{Theorem}[section]
\newtheorem{theorem}[theo]{Theorem}

\newtheorem{lemma}[theo]{Lemma}
\newtheorem{definition}[theo]{Definition}

\theoremstyle{empty}
\theorembodyfont{\rmfamily}
\newtheorem{refproof}{theorem}

\journal{Theoretical Computer Science}

\begin{document}

\begin{frontmatter}

\title{Spectra, hitting times, and resistance distances of $q$-subdivision graphs}

\author[label01,label02]{Yibo Zeng}
\author[label01,label03]{Zhongzhi Zhang}
\ead{zhangzz@fudan.edu.cn}
\address[label01]{Shanghai Key Laboratory of Intelligent Information
Processing, Fudan University, Shanghai 200433, China}
\address[label02]{School of Mathematical Sciences, Fudan
University, Shanghai 200433, China}
\address[label03]{School of Computer Science, Fudan
University, Shanghai 200433, China}

%%%%%%%%%%%%%%%%%%%%%%%%%%%%%%%%%%%%%%%%%%%%
%Subdivision, triangulation, Kronecker product, corona product and many other
\begin{abstract}
Graph operations or products play an important role in complex networks.  In this paper, we study the properties of $q$-subdivision graphs, which have been applied to model complex networks. For a simple connected graph $G$, its $q$-subdivision graph $S_q(G)$ is obtained from $G$ through replacing every edge $uv$ in $G$ by $q$ disjoint paths of length 2, with each path having $u$ and $v$ as its ends. We derive explicit formulas for many quantities of $S_q(G)$ in terms of those corresponding to $G$, including the eigenvalues and eigenvectors of normalized adjacency matrix, two-node hitting time, Kemeny constant,  two-node resistance distance, Kirchhoff index, additive degree-Kirchhoff index, and multiplicative degree-Kirchhoff index. We also study the properties of the iterated $q$-subdivision graphs, based on which we obtain the closed-form expressions for a family of hierarchical lattices, which has been used to describe scale-free fractal networks.
\end{abstract}


\end{frontmatter}

%%%%%%%%%%%%%%%%%%%%%%%%%%%%%%%%%%%%%%%%%%%%
%%%%%%%%%%%%%%%%%%%%%%%%%%%%%%%%%%%%%%%%%%%%
%%%%%%%%%%%%%%%%%%%%%%%%%%%%%%%%%%%%%%%%%%%%
%%%%%%%%%%%%%%%%%%%%%%%%%%%%%%%%%%%%%%%%%%%%

\section{Introduction}

As powerful tools of network science, graph operations and products  have been widely used to construct complex networks with the remarkable scale-free~\cite{BaAl99}, small-world~\cite{WaSt98}, and fractal~\cite{SoHaMa05} characteristics observed in realistic networks~\cite{Ne03}. A clear advantage for generating complex networks by graph operations and products lies in the allowance of rigorous analysis for  structural and dynamical properties of the resulting networks. In addition, various real massive networks comprise of smaller pieces, such as communities~\cite{GiNe02}, motifs~\cite{MiShItKaKhAl02}, and cliques~\cite{Ts15}. Graph operations and products represent a natural way to create a huge graph out of small ones. Due to the great relevance, diverse graph operations and products have been introduced or developed  for practical purposes, e.g. designing models for complex networks. Frequently used graph operations and products include edge iteration~\cite{DoGoMe02,ZhCo11}, planar triangulation~\cite{AnHeAnDa05,DoMa05,JiLiZh17}, Kronecker product~\cite{We62,LeFa07,LeChKlFaGh10}, hierarchical product~\cite{BaCoDaFi09,BaDaFiMi09,BaCoDaFi16}, and corona product~\cite{LvYiZh15,ShAdMi17,QiLiZh18}.

Among various graph operations and products, subdivision is one of the most popular ones. For a simple graph $G$, its subdivision graph is the graph obtained from $G$ by inserting a new node into every edge of $G$. The properties of subdivision graphs have been extensively studied~\cite{Wo05, HuYaQi15,CaMiMo16}. Moreover, many extended subdivision graphs were proposed, such as $q$-full subdivision graph~\cite{We01,FiHa12} and $q$-subdivision graphs~\cite{Ya88}. The $q$-full subdivision graph of $G$ is obtained from $G$ by replacing each of its edges with pairwise internally disjoint paths of length $q+1$, while the $q$-subdivision graph $S_q(G)$ of $G$ is obtained from $G$ by replacing each edge $uv$ with $q$ disjoint paths of length $2$: $ux_1v$, $ux_2v$, $\ldots$, $ux_qv$. The $q$-subdivision operation was iteratively applied to the particular graph consisting of an edge,  generating the hierarchical lattices---a model of complex networks with the striking scale-free fractal topologies~\cite{ZhZhZo07}, which has received much recent attention~\cite{ZhYaGa11,ZhShHuCh12,LiZh17}. However, in contrast to the traditional subdivision, the properties of $S_q(G)$ for a general graph  $G$ are still not well understood, despite the wide ranges of applications for this graph operation.

In this paper, we present an extensive study of the properties for $q$-subdivision graph $S_q(G)$ of a simple connected graph $G$. We provide explicit formulas for eigenvalues and eigenvectors of normalized adjacency matrix for $S_q(G)$ in terms of those associated with $G$, based on which we determine two-node hitting time and the Kemeny constant for random walks on $S_q(G)$ in terms of those corresponding to $G$. Also, we derive the expressions of two-node resistance distance, Kirchhoff index, additive degree-Kirchhoff index for $S_q(G)$, and multiplicative degree-Kirchhoff index, in terms of the quantities for $G$. Finally, we obtain closed-form solutions to related quantities for iterated $q$-subdivisions of a graph $G$, and apply those obtained results to  the scale-free fractal hierarchical lattices,  leading to  explicit expressions for some quantities.

\section{Preliminaries}

In this section, we introduce some basic concepts for a graph, random walks and electrical networks.

\subsection{Graph and Matrix Notation}

%To start with, for the completeness of our work, we may first shed light on several basic concepts.

Let $G(V,E)$ be a simple connected graph with $n$ nodes and $m$ edges.  The $n$ nodes constitute node set $V(G)=\{1,2,\ldots,n\}$, and $m$ edges form edge set $E(G)=\{e_1,e_2,\ldots, e_m\}$.

Let $A$ denote the adjacency matrix of $G$, the entry $A(i,j)$ of which is 1 (or 0) if nodes $i$ and $j$ are (not) directly connected in $G$. Let $\Gamma (i)$ denote the set of neighbors of  node $i$  in graph $G$. Then the degree of node $i$ is  $d_i=\sum_{j  \in \Gamma (i)}A(i,j)$, which constitutes the $i$th entry of the diagonal degree matrix $D$ of $G$. The incidence matrix of  $G$ is an $n\times m$ matrix $B$, where $B(i,j)=1$ (or 0) if $i$ is (not) incident with $e_j$. %Then the degree of node $i$ is given by $d_i={\sum{^{n}_{j=1}}}A(i,j)$,

\begin{lemma}\label{b1a}\emph{\cite{CvDoSa80}}
Let $G$ be a simple connected graph with $n$ nodes. Then the rank of its incidence matrix $B$ is
            ${\rm rank} (B)=n-1$ if $G$ is bipartite, and ${\rm rank} (B)=n$  otherwise.
\end{lemma}

\subsection{Random Walks on Graphs}

For a graph $G$, we can define a discrete-time unbiased random walk taking place on it. For any time step, the walker jumps from its current location, node $i$,  to another node $j$ with probability $A(i,j)/d_i$. Such a random walk on $G$ is in fact a Markov chain characterized by the transition probability matrix $T=D^{-1}A$, with the entry  $T(i,j)$ equal to $A(i,j)/d_i$. For a random walk on graph $G$, the stationary distribution is an $n$-dimension vector $\pi=(\pi_1, \pi_2, \ldots, \pi_n)$ satisfying $\pi  T=\pi$ and $\sum_{i=1}^n \pi_i=1 $. It is easy to verify that $\pi=(d_1/2m, d_2/2m, \ldots, d_n/2m)$  for unbiased random walks on $G$.
	
The transition  probability matrix $T$ of graph $G$ is not symmetric. However, $T$ is similar to the normalized adjacency matrix $P$ of $G$, which is defined by
\begin{equation*}
P=D^{-\frac{1}{2}} A D^{-\frac{1}{2}}=D ^{\frac{1}{2}}T D^{-\frac{1}{2}}\,.
\end{equation*}
Obviously, $P$ is symmetric, with the $(i,j)$th entry being $P(i,j)=\frac{A(i,j)}{\sqrt{d_id_j}}$.

\begin{lemma}\label{b1}\emph{\cite{Ch97}}
Let $G$ be a simple connected graph with $n$ nodes, and let
$1=\lambda_1>\lambda_2\geq\ldots\geq\lambda_n\geq-1$ be  the eigenvalues of its normalized adjacency matrix $P$. Then
$\lambda_n=-1$ if and only if $G$ is bipartite.
\end{lemma}

Let $v_1, v_2,\ldots, v_n$ be the normalized mutually orthonormal eigenvectors corresponding to  the $n$ eigenvalues $\lambda_1,\lambda_2,\ldots,\lambda_n$, where $v_i=(v_{i1},v_{i2},\ldots,v_{in})^\top$. Then,
\begin{equation}\label{eig=1}
  v_1=\Big(\sqrt{d_1/2m},\sqrt{d_2/2m},...,\sqrt{d_n/2m}\Big)^\top
\end{equation}
and
\begin{equation}
\sum_{k=1}^n v_{ik}v_{jk}=\sum_{k=1}^n v_{ki}v_{kj}=\left\{
                                                      \begin{array}{ll}
                                                        1, & \hbox{if $i=j$;} \\
                                                        0, & \hbox{otherwise.}
                                                      \end{array}
                                                    \right.
\end{equation}
As for a bipartite graph $G$, whose node set $V(G)$ can be divided into two disjoint subsets $V_1$ and $V_2$, i.e., $V(G)=V_1\cup V_2$, we have
\begin{equation}\label{eig=-1}
  v_{ni}=\sqrt{d_i/2m},\, i\in V_1;~~v_{nj}=-\sqrt{d_j/2m}, \,  j\in V_2.
\end{equation}

A fundamental quantity related to random walks is the hitting time. The hitting time $T_{ij}$ from one node $i$ to another node $j$ is the expected time taken by a walker to first reach node $j$ starting from node $i$, which is relevant in various scenarios~\cite{Re01}. Many interesting quantities of graph $G$ can be defined or derived from hitting times.  For example, for a  graph $G$, its Kemeny's constant $K(G)$  is defined as  the expected number of steps required for a walker starting from node $i$ to a destination node, which is chosen randomly according to a stationary distribution of  random walks on $G$~\cite{Hu14}.  The Kemeny's constant $K(G)$  is independent of the selection of starting node $i$~\cite{LeLo02}.

The hitting time  $T_{ij}$ for random walks on graph $G$ is encoded in the eigenvalues and eigenvectors of its normalized adjacency matrix $P$.
\begin{theorem}\label{Theorem1}\emph{\cite{Lo93}}
For random walks on a simple connected graph $G$, the  hitting time $T_{ij}$ from one node $i$ to another node $j$ is
\begin{equation*}
  T_{ij}=2m\sum_{k=2}^n \frac{1}{1-\lambda_k}
  \left(\frac{v_{kj}^2}{d_j}-\frac{v_{ki}v_{kj}}{\sqrt{d_i d_j}}\right).  %, \forall i,j \in V(G)
\end{equation*}
In particular, when  $G$ is a bipartite graph with $V(G)=V_1 \cup V_2$, then
\begin{equation*}
T_{ij}= 2m \sum_{k=2}^{n-1} \frac{1}{1-\lambda_k}
     \left(\frac{v_{kj}^2}{d_j}-\frac{v_{ki}v_{kj}}{\sqrt{d_id_j}}\right),
\end{equation*}
if $i$ and $j$ are both in $V_1$ or $V_2$;
\begin{equation*}
T_{ij}= 2m \sum_{k=2}^{n-1} \frac{1}{1-\lambda_k}
      \left(\frac{v_{kj}^2}{d_j}-\frac{v_{ki}v_{kj}}{\sqrt{d_id_j}}\right)+1,
\end{equation*}
otherwise.
\end{theorem}

In contrast, the Kemeny's constant of $G$ is only dependent on the eigenvalues of $P$.
\begin{lemma}\emph{\cite{butler2016algebraic}}\label{lemmaKem}
Let $G$ be a simple connected graph with $n$ nodes. Then
\begin{equation*}
K(G)=\sum_{j=1}^{n} \pi_j T_{ij}= \sum_{i=2}^n\frac{1}{1-\lambda_i},
\end{equation*}
where $1=\lambda_1>\lambda_2\geq\ldots\geq\lambda_n\geq-1$ are eigenvalues of matrix $P$.
\end{lemma}

\subsection{Electrical Networks}

For a simple connected graph $G$, we can define a corresponding electrical network $G^*$,  which is obtained from   $G$ by replacing each edge in  $G$ with  a unit resistor~\cite{DoSn84}. The resistance distance $r_{ij}$ between a pair of nodes $i$ and $j$ in $G$ is equal to the effective resistance between $i$ and $j$ in $G^*$.  Similar to the hitting time $T_{ij}$, resistance distance
$r_{ij}$ can also be expressed in terms of the eigenvalues and eigenvectors of normalized adjacency matrix $P$.
\begin{lemma}\emph{\cite{ChZh07}} \label{lemmaRij}
Let $G$ be a simple connected  graph.  Then   resistance distance $r_{ij}$ between nodes $i$  and $j$ is
\begin{equation*}
 r_{ij}=\sum_{k=2}^n \frac{1}{1-\lambda_k}
  \left(\frac{v_{ki}}{\sqrt{d_i }}-\frac{v_{kj}}{\sqrt{d_j}}\right)^2.  %, \forall i,j \in V(G)
\end{equation*}
\end{lemma}

\begin{lemma}\label{Foster}
\emph{\cite{Fo49}} Let $G$ be a simple connected graph with $n$ nodes. Then the sum of resistance distances between all pairs of adjacent nodes in  $G$ is equivalent to $n-1$, i.e.
    \begin{equation*}
        \sum_{ij\in E(G)}r_{ij}=n-1.
    \end{equation*}
where the summation is taken over all the edges in $G$.
\end{lemma}

There are some intimate relationships between random walks on graphs and electrical networks. For example,
the resistance distance $r_{ij}$ is closely related to hitting times $T_{ij}$ and $T_{ji}$ of $G$, as stated in the  following lemma.
\begin{lemma}\emph{\cite{ChRaRuSm89}}\label{lemmaR}
For any pair of nodes $i$ and $j$ in a graph $G$ with $m$ edges, the following relation holds true:
\begin{equation*}
2mr_{ij}=T_{ij}+T_{ji}. %,~\forall i,j \in V(G).
\end{equation*}
\end{lemma}

The resistance distance is an important quantity~\cite{GhBoSa08}. Various graph invariants based on resistance distances have been defined and studied. Among these invariants, the Kirchhoff index~\cite{KlRa93} is of vital importance.
\begin{definition}\emph{\cite{KlRa93}}\label{defK}
The Kirchhoff index of a graph $G$ is defined as
\begin{equation*}
 \mathcal{K}(G)=\frac{1}{2}\sum_{i,j=1}^{n}r_{ij}=\sum_{\{i,j\}\subseteq V(G)}r_{ij}.
\end{equation*}
\end{definition}
Kirchhoff index  has found wide applications. For example, it can be used as  measures of the overall connectedness of a network~\cite{TiLe10}, the robustness of first-order consensus algorithm in noisy networks~\cite{PaBa14}, as well as the edge centrality of complex networks~\cite{LiZh18}.
In recent years,  several modifications for Kirchhoff index have been proposed, including additive degree-Kirchhoff index~\cite{GuFeYu12} and multiplicative degree-Kirchhoff index~\cite{ChZh07}. For a graph $G$, its  additive degree-Kirchhoff index $ \bar{\mathcal{K}}(G)$ and multiplicative degree-Kirchhoff index  $\tilde{\mathcal{K}}(G)$ are defined as
\begin{equation*}
  \bar{\mathcal{K}}(G)=\frac{1}{2}\sum_{i,j=1}^{n}(d_i+d_j)r_{ij}=\sum_{\{i,j\}\subseteq V(G)}(d_i+d_j)r_{ij}
\end{equation*}
and
\begin{equation*}
    \tilde{\mathcal{K}}(G)=\frac{1}{2}\sum_{i,j=1}^{n}d_id_jr_{ij}=\sum_{\{i,j\}\subseteq V(G)}d_id_jr_{ij},
\end{equation*}
respectively.

It has been proved that $\tilde{\mathcal{K}}(G)$ can be represented in terms of  the eigenvalues of the matrix $P$.
\begin{lemma}\emph{\cite{ChZh07}}\label{lemmaK*}
Let $G$ be a graph with $n$ nodes and $m$ edges. Then
\begin{equation*}
 \tilde{\mathcal{K}}(G)=2m\sum_{i=2}^n\frac{1}{1-\lambda_i}.
\end{equation*}
%where $1=\lambda_1>\lambda_2\geq \ldots \geq\lambda_n\geq-1$ are eigenvalues of the matrix $P$.
\end{lemma}

\section{ $q$-subdivision Graphs and Their Matrices }

In this section, we introduce the  $q$-subdivision graph of a graph $G$,  which is an extension of the traditional subdivision graph, since $1$-subdivision graph is exactly the  subdivision graph.  The subdivision of $G$, denoted by $S(G)$, is the graph obtained from $G$ by inserting a new node into  each edge in $G$.  The subdivision graph can be easily extended to a general case.

\begin{definition}
Let $G$ be a simple connected graph. For a positive integer $q$, the $q$-subdivision graph of $G$, denoted by $S_q(G)$, is the graph obtained from $G$ by  replacing each edge $uv$ in $G$ with $q$ disjoint paths of length $2$: $ux_1v$, $ux_2v$, $\ldots$, $ux_qv$.
\end{definition}

In what follows, for a quantity $Z$ of $G$, we use   $\hat{Z}$  to denote the corresponding quantity associated with $S_q(G)$. Then it is easy to verify that in the $q$-subdivision graph $S_q(G)$, there are $\hat{n}=n+mq$ nodes and $\hat{m}=2mq$ edges.

By definition, $S_q(G)$ is a bipartite graph, irrespective of $G$.  Then, the  node set $\hat{V}:=V(S_q(G))$ of $ S_q(G)$ can be divided into two disjoint parts $V$ and $V^\prime$,  where $V$ is the set of old nodes inherited from $G$, while $V^\prime$ is the set of  new nodes generated in the process of performing $q$-subdivision operation on $G$. Moreover, $V^\prime$ can be further classified into $q$ parts as ${V}^\prime=V^{(1)} \cup  V^{(2)}\cup \cdots  \cup V^{(q)}$,  where each   $V^{(i)}$ ($i=1,2,\ldots,q$) contains $m$ new nodes produced by $m$ different edges in $G$. Namely,
\begin{equation} \label{div1}
\hat{V} =V \cup V^{(1)} \cup  V^{(2)}\cup \cdots  \cup V^{(q)}.
\end{equation}
%Let $\Gamma_{S_q(G)}(x)$ denote the set of neighbors of  node $x$  in graph $S_q(G)$.
By construction, for each old edge $uv$, there exists one and only one node $x$ in each $V^{(i)}$ ($i=1,2,\ldots,q$),  satisfying  $\hat{\Gamma} (x)=\{u,v\}$. Thus, for two different sets $V^{(i)}$ and $V^{(j)}$, the structural and dynamical properties of nodes belonging them are  equivalent to each other.

For $S_q(G)$,  its adjacency matrix $\hat{A}$,  diagonal degree matrix $\hat{D}$, and normalized adjacency matrix $\hat{P}$, can be expressed in terms of related matrices of $G$ as
\begin{equation*}
\hat{A}=
\left(
  \begin{array}{cccc}
    O           & B         & \cdots\   & B \\
    B^\top      & O         & \cdots\   & O  \\
    \vdots\     & \vdots\   & \ddots\   & \vdots\ \\
    B^\top      & O         &\cdots\    & O  \\
  \end{array}
\right),
\end{equation*}
\begin{equation*}
 \hat{D}={\rm diag} \{qD,\underbrace{2I_m,\ldots,2I_m}_q\},
\end{equation*}
\begin{equation}\label{Matrix P}
\hat{P}=\hat{D}^{-\frac{1}{2}} \hat{A} \hat{D}^{-\frac{1}{2}}
=\frac{1}{\sqrt{2q}}
\left(
  \begin{array}{cccc}
    O                       &D^{-\frac{1}{2}}B  & \cdots\   & D^{-\frac{1}{2}}B \\
    B^\top D^{-\frac{1}{2}} & O                 & \cdots\   & O \\
    \vdots\                 & \vdots\           & \ddots\   & O \\
    B^\top D^{-\frac{1}{2}} & O                 & \cdots\   & O \\
  \end{array}
\right),
\end{equation}
where $I_m$ is the $m \times m $ identity matrix.

\section{Eigenvalues and Eigenvectors of Normalized Adjacency Matrix  for  $q$-subdivision Graphs}

In this section, we study the eigenvalues and eigenvectors of normalized adjacency matrix $\hat{P}$  for  $q$-subdivision graphs $S_q(G)$. We will show that both eigenvalues and eigenvectors for  $\hat{P}$ can be expressed in terms of those related quantities associated with graph $G$.

\begin{lemma}\label{l0}
Let $\hat\lambda$ be any non-zero eigenvalue of $\hat{P}$. Then, after appropriate labeling of  nodes,  the eigenvector $\psi$ corresponding  to $\hat\lambda$ can  be rewritten in the form $\psi^\top=(\psi_1,\psi_2,\ldots,\psi_{n+qm})^\top=\left( {\psi^\prime}^\top,{\psi^{(1)}}^\top,{\psi^{(2)}}^\top,\cdots,{\psi^{(q)}}^\top\right )$, where $\psi^\prime$ is an $n$ dimensional vector, $\psi^{(i)}$ ($i=1,2,\ldots,q$) is an $m$ dimensional vector satisfying $\psi^{(1)}={\psi^{(2)}}=\cdots=\psi^{(q)}$.
\end{lemma}
\begin{proof}
By definition of eigenvalues and eigenvectors, we have  $ \hat\lambda\psi=\hat{P}\psi$. Considering   Eq.~\eqref{Matrix P}, we obtain $\hat{\lambda}\psi^{(i)}=\frac{1}{\sqrt{2q}}B^\top D^{-\frac{1}{2}}\psi^\prime$, $i=1,2,\ldots,q$. When  $\hat\lambda\neq0$,  $\psi^{(i)}=\frac{1}{\hat{\lambda}\sqrt{2q}}B^\top D^{-\frac{1}{2}}\psi^\prime$ holds for all $i=1,2,\ldots,q$.
\end{proof}

Now we are ready to evaluate the full eigenvalues and their multiplicities of $\hat{P}$.
\begin{lemma}\label{l1}
Let $\hat\lambda$ be a  non-zero eigenvalue of $\hat{P}$. Then, $2\hat\lambda^2-1$ is an eigenvalue of $P$ and its multiplicity, denoted by $m_P(2\hat\lambda^2-1)$ is identical to  the multiplicity $m_{\hat P}(\hat\lambda)$ for  eigenvalue  $\hat\lambda$ of $\hat P$.
\end{lemma}

\begin{proof}
Let $\hat\lambda$ be an eigenvalue of $\hat{P}$, and let $\psi^\top=(\psi_1,\psi_2,\ldots,\psi_{n+mq})^\top$ be its corresponding eigenvector. In addition, let $\psi^\prime$ be an $n$-dimensional vector obtained from $\psi$ by restricting its components  to the node set $V$. Then $\psi^\prime$ is an eigenvector of $P$ for $G$, as we will show below. By definition,
\begin{equation} \label{Eig1}
 \hat\lambda\psi=\hat{P}\psi.
\end{equation}

Our goal is to express $\hat{\lambda}$ in terms of eigenvalues of $P$. For this purpose, we consider an old node $x$ in $\hat{P}$. Let  $d_x$ and $\hat{d}_x$ denote  the degree of node $x$ in graphs $G$ and $S_q(G)$, respectively.  Then, by construction, we have the following relation $\hat{d}_x=q\, d_x$. From Eq.~\eqref{div1},  the neighbors of $x$  can be divided  into $q$ classes, which are, respectively, in $V^{(1)}$, $V^{(2)} $, $\ldots$ , $V^{(q)}$.  Moreover, the properties of nodes in  these $q$ classes are identical. We can appropriately label the nodes in  $V$ such that $\{ 1, 2,\ldots,  d_x\}$ is the set of neighbors for $x$ in $G$.   Then in $S_q(G)$  %For each integer $k$ ($0 \leq k\leq q-1$),
we assume that  nodes with labelling  $n+km+i$ ($0\leq k \leq q-1$ and $ 1\leq i\leq d_x$) are neighbors of $x$, which   belong to  $V^{(k+1)}$. Note that in $S_q(G)$   each new neighbor of $x$ is simultaneously connected to an old neighbor of $x$. For the sake of convenience, we  assume that the neighbors of the newly-added node $n+km+i$ ($0\leq k \leq q-1$ and $ 1\leq i\leq d_x$) are $x$ and $i$. %$\hat{V}$

From Eq.~\eqref{Eig1}, we  obtain the equation corresponding to node $x$, which reads
\begin{equation} \label{Eig2}
\begin{aligned}
  \hat\lambda\psi_x         &= \sum_{j=1} ^{\hat{d}_x} \hat{P}(x,j) \psi_j
                             = \sum_{k=0} ^{q-1} \sum_{i=1} ^{d_x} \hat{P}(x,n+km+i) \psi_{n+km+i}
\\                          &= q \sum_{i=1} ^{d_x} \hat{P}(x,n+i) \psi_{n+i},
\end{aligned}
\end{equation}
where the next-to-last equality is obtained according to Lemma~\ref{l0}.
Analogously, for a newly-added node $s$ with neighboring set $\hat{\Gamma}(s)=\{u,v\}$, we obtain %their corresponding elements of eigenvector $\psi$ satisfy
\begin{equation}\label{Eig3}
 \hat\lambda \psi_{s}= \hat P(s,u)\psi_u +\hat P(s,v)\psi_v,
\end{equation}
which implies  % Moreover, considering $\hat\lambda\neq0$, we can have the following results:
\begin{equation}\label{Eig4}
 \psi_{s}= \frac{\hat P(s,u)}{\hat\lambda}\psi_u +\frac{\hat P(s,v)}{\hat\lambda}\psi_v.
\end{equation}
Thus, $ \psi_{n+i}$ in  Eq.~\eqref{Eig1} can be written as
\begin{equation}\label{Eig5}
 \psi_{n+i}= \frac{\hat{P}(n+i,x)}{\hat\lambda}\psi_x +\frac{\hat{P}(n+i,i)}{\hat\lambda}\psi_i.
\end{equation}
Inserting Eq.~\eqref{Eig5}  into Eq.~\eqref{Eig2}  leads to
\begin{equation}\label{Eig6}
\hat\lambda \psi_x = q \sum_{i=1} ^{d_x} \frac{\Big(\hat{P}(x,n+i)\Big)^2}{\hat\lambda}\psi_{x}
                         + q \sum_{i=1} ^{d_x} \frac{\hat{P}(n+i,x) \hat{P}(n+i,i)}{\hat\lambda}\psi_{i}.
\end{equation}
By definition of $\hat{P}$, for each neighbour $n+i$ of node $x$, one has
\begin{equation}\label{Eig7}
 \hat{P}(x,n+i)=\frac{1}{\sqrt{2\hat{d}_x}}=\frac{1}{\sqrt{2qd_x}}.
\end{equation}
In addition, for an old neighbor $i$ of node $x$, the following relation holds:
\begin{equation}\label{Eig8}
\begin{aligned}
 \hat{P}(n+i,x) \hat{P}(n+i,i)  &   = \frac{1}{\sqrt{2\hat{d}_i}} \frac{1}{\sqrt{2\hat{d}_x}}
                                     \\ &   = \frac{1}{2} \frac{1}{\sqrt{\hat{d}_i \hat{d}_x}}
                                            = \frac{1}{2} \hat{P}(x,i).
\end{aligned}
\end{equation}
Substituting Eqs.~\eqref{Eig7} and~\eqref{Eig8}  back into Eq.~\eqref{Eig6}, we arrives at
\begin{equation}\label{Eig9}
   \left (\hat\lambda^2-\frac{1}{2}\right)\psi_x=\frac{q}{2} \sum_{i=1}^{d_x} \hat{P}(x,i)\psi_i,
\end{equation}
which only involves old nodes in $V$.
Thus, according to the following relation
\begin{equation}\label{Eig10}
    \hat{P}(x,i)=\frac{1}{\sqrt{\hat{d}_x \hat{d}_i}}=\frac{1}{q\sqrt{d_x d_i}}=\frac{P(x,i)}{q},
\end{equation}
we have
\begin{equation}\label{Eig11}
    (2\hat\lambda^2-1)\psi_ x= \sum_{i=1}^{d_x} P(x,i) \psi_i,~\forall x\in V,
\end{equation}
implying that $2\hat\lambda^2-1$ is an eigenvalue of $P$, and $\psi^\prime$, the $n$-dimensional restricted vector defined above, is one associated eigenvector. Furthermore, $\psi$ can be totally determined by $\psi^\prime$ using Eq.~\eqref{Eig4}. Thus $m_P(2\hat\lambda^2-1)\geq m_{\hat P}(\hat\lambda)$.

Suppose that $m_P(2\hat\lambda^2-1) > m_{\hat P}(\hat\lambda)$. This means that there should exist an extra eigenvector $\psi_e$ associated to $2\hat\lambda^2-1$ without a corresponding eigenvector in $\hat{P}$. But Eq.~\eqref{Eig4} provides $\psi_e$ with an associated eigenvector of $\hat{P}$ since $\lambda\neq0$. This contradicts our assumption. Therefore, $m_P(2\hat\lambda^2-1) = m_{\hat P}(\hat\lambda)$.
\end{proof}

\begin{lemma} \label{l2}
Let $\lambda$ be any eigenvalue of $P$ such that $\lambda \neq -1$. Then $\sqrt{\frac{1+\lambda}{2}}$ and $-\sqrt{\frac{1+\lambda}{2}}$ are eigenvalues of $\hat{P}$ and
 $m_{\hat{P}}\left(\sqrt{\frac{1+\lambda}{2}}\right) = m_{\hat{P}}\left(-\sqrt{\frac{1+\lambda}{2}}\right)
 =m_P(\lambda)$.
\end{lemma}
\begin{proof}
This is a direct consequence of Lemma~\ref{l1}.
\end{proof}

Lemmas~\ref{l1} and~\ref{l2} show that all nonzero eigenvalues and their corresponding eigenvectors  of $\hat{P}$ can be obtained from those of $P$.
For  those zero eigenvalues and their associated eigenvectors of $\hat{P}$, we can characterize them easily.   %$\psi_i^{(j)}$ represents the $j$-th component of the $i$-th eigenvector in the following theorem.

\begin{theorem}\label{Theorem2}
Let $G$ be a simple connected graph with $n$ nodes and $m$ edges. Let $1=\lambda_1>\lambda_2\geq\cdots\geq\lambda_n\geq-1$ be the eigenvalues of $P$, and let $v_1,v_2\ldots,v_n$ be their corresponding orthonormal eigenvectors. Then
\begin{enumerate}[(i)]
\item if $G$ is non-bipartite, then $\pm\sqrt{\frac{1+\lambda_i}{2}}$, $i=1,2,\ldots,n$, are eigenvalues of $\hat{P}$, and the element of  their orthonormal eigenvectors corresponding to node $j$ is
 \begin{equation*}
     \left\{
                                    \begin{array}{ll}
                                      \frac{1}{\sqrt{2}}v_{ij}, & \hbox{$j\in V$,} \\
                                      \pm\sqrt{\frac{1}{2q(1+\lambda_i)}}\left(\frac{v_{is}}{\sqrt{d_s}}+\frac{v_{it}}{\sqrt{d_t}}\right),
                                      & \hbox{$j\in V^\prime$, $\hat{\Gamma} (j)=\{s,t\}$;}
                                    \end{array}
                                  \right.
      \end{equation*}
 and $0$'s are eigenvalues of $\hat{P}$ with multiplicity $mq-n$, with  their corresponding orthonormal eigenvectors being
       \begin{equation*}
             \\ \left(
                  \begin{array}{c}
                    0 \\
                    Y_z \\
                  \end{array}
                \right) , \, z=1,2,\ldots, mq-n,
      \end{equation*}
      where $Y_1$, $Y_2$, $\ldots$, $Y_{mq-n}$ are an orthonormal basis of the kernel space of matrix
      \begin{equation*}
      C:={\underbrace{\left(
        \begin{array}{cccc}
            B & B & \cdots\ & B  \\
        \end{array}
      \right)}_q}.
      \end{equation*}
  \item if $G$ is non-bipartite, then $\pm\sqrt{\frac{1+\lambda_i}{2}}$, $i=1,2,\ldots,n-1$, are eigenvalues of $\hat{P}$, and the element of  their orthonormal eigenvectors corresponding to node $j$ is
    \begin{equation*}
     \left\{
                                    \begin{array}{ll}
                                      \frac{1}{\sqrt{2}}v_{ij}, & \hbox{$j\in V$,} \\
                                      \pm\sqrt{\frac{1}{2q(1+\lambda_i)}}\left(\frac{v_{is}}{\sqrt{d_s}}+\frac{v_{it}}{\sqrt{d_t}}\right),
                                      & \hbox{$j\in V^\prime$, $\hat{\Gamma}(j)=\{s,t\}$;}
                                    \end{array}
                                  \right.
      \end{equation*}
   and $0$'s are eigenvalues of $\hat{P}$ with multiplicity $mq-n+2$, with  their corresponding orthonormal eigenvectors being
         \begin{equation*}
             \\ \left(
                  \begin{array}{c}
                    v_n \\
                    0 \\
                  \end{array}
                \right),
              \\ \left(
                  \begin{array}{c}
                    0 \\
                    Y_z \\
                  \end{array}
                \right) , z=1,2,\ldots,mq-n+1;
      \end{equation*}
      where $Y_1$,$Y_2$, $\ldots$, $Y_{mq-n+1}$ are an orthonormal basis of the kernel space of matrix
      \begin{equation*}
      C:={\underbrace{\left(
        \begin{array}{cccc}
            B & B & \cdots\ & B  \\
        \end{array}
      \right)}_q}.
      \end{equation*}
\end{enumerate}
\end{theorem}
\begin{proof}
We first prove (i). Since $G$ is non-bipartite, by Lemma~\ref{b1}, every  eigenvalue $\lambda_i$ of $P$ is not equal to $-1$. According to Lemma~\ref{l2}, we can obtain all nonzero eigenvalues $\pm\sqrt{\frac{1+\lambda_i}{2}}$ of $\hat{P}$ and their multiplicity. Moreover, using Eq.~\eqref{Eig4}, the element of their  associated  eigenvectors corresponding to node $j$ is
\begin{equation*}
    \left\{
                                    \begin{array}{ll}
                                      v_{ij}, & \hbox{$j\in V$,} \\
                                      \pm\sqrt{\frac{1}{q(1+\lambda_i)}}\left(\frac{v_{is}}{\sqrt{d_s}}+\frac{v_{it}}{\sqrt{d_t}}\right),
                                      & \hbox{$j\in V^\prime$, $\hat{\Gamma}(j)=\{s,t\}$,}
                                    \end{array}
                                  \right.
      \end{equation*}
which can be be orthonormalized to obtain the orthonormal eigenvectors.

%We have extensively known $2n$ eigenvalues and eigenvectors of $\hat{P}$. The rest of the eigenvalues are supposed to be $0$s.
For the zero eigenvalues, from Lemma~\ref{b1a}, ${\rm rank}(B)=n$ since $G$ is non-bipartite. Thus, ${\rm rank}(C)=n$ and $\dim({\rm Ker}(C))=mq-n$. Let $Y_1$, $Y_2$, $\ldots$, $Y_{mq-n}$ be an orthonormal basis of the kernel space of matrix $C$. It is easy to confirm  that
     $\left(
         \begin{array}{c}
                    0 \\
                    Y_z \\
         \end{array}
     \right)$, $z=1,2,\ldots, mq-n$, are eigenvectors for eigenvalues $0$ of matrix $\hat{P}$.

For (ii), our proof is similar. We just need to verify that
\begin{equation*}
             \\ \hat{P}\left(
                  \begin{array}{c}
                    v_n \\
                    0 \\
                  \end{array}
                \right)=O_{(n+mq)\times1},
\end{equation*}
which is trivial according to Eq.~\eqref{Matrix P}.
\end{proof}

Note that when $q=1$, Theorem~\ref{Theorem2} coincides with result in~\cite{XiZhC16a}.

\section{Hitting Times for Random Walks on $q$-subdivision Graphs}

Theorem~\ref{Theorem2} provides  complete information about the eigenvalues and eigenvectors of $\hat{P}$ in terms of those $P$. In this section, we use this information to determine two-node hitting time and Kemeny constant for unbiased random walks on $S_q(G)$.

\subsection{Two-Node Hitting Time}

We first compute the hitting time from one node to another in $S_q(G)$. To this end, we  express the orthonormal eigenvectors of $S_q(G)$ in more explicit forms. By Eqs.~(\ref{eig=1})~(\ref{eig=-1}) and Theorem~\ref{Theorem2}, we can directly derive the following results.

\begin{enumerate}[(i)]
  \item The eigenvectors corresponding to eigenvalues  $\pm\sqrt{\frac{1+\lambda_1}{2}}=\pm1$ for matrix $\hat{P}$ are
  \begin{equation}\label{re1A}
    \left(\sqrt{\frac{d_1}{4m}},\cdots,\sqrt{\frac{d_n}{4m}},
      \sqrt{\frac{1}{2mq}},\cdots,\sqrt{\frac{1}{2mq}} \right)^\top
  \end{equation}
  and
  \begin{equation}\label{re1B}
         \left(\sqrt{\frac{d_1}{4m}},\cdots,\sqrt{\frac{d_n}{4m}},
      -\sqrt{\frac{1}{2mq}},\cdots,-\sqrt{\frac{1}{2mq}}\right)^\top,
  \end{equation}
  respectively.
    \item If $G$ is non-bipartite, for each $j\in V^\prime$ with $\hat{\Gamma}(j)=\{s,t\}$,
    \begin{equation}\label{sum1}
        \sum_{z=1}^{mq-n} Y^2_{zj}=1-\frac{1}{mq}-\sum_{k=2}^n \frac{1}{(1+\lambda_k)q}\left(\frac{v_{ks}}{\sqrt{d_s}}+\frac{v_{kt}}{\sqrt{d_t}}\right)^2 .
    \end{equation}

  \item If $G$ is bipartite,  for each $j \in V^\prime$ with $\hat{\Gamma}(j)=\{s,t\}$,
    \begin{equation}\label{sum2}
        \sum_{z=1}^{mq-n+1} Y^2_{zj}=1-\frac{1}{mq}-\sum_{k=2}^{n-1} \frac{1}{(1+\lambda_k)q}\left(\frac{v_{ks}}{\sqrt{d_s}}+\frac{v_{kt}}{\sqrt{d_t}}\right)^2.
    \end{equation}
\end{enumerate}

Now we present our results for hitting times of random walks on $S_q(G)$.
\begin{theorem}\label{HT}
Let $G$ be a simple connected graph with $n$ nodes and $m$ edges. $S_q(G)$ is the $q$-subdivision graph of $G$ with $\hat{V}=V\cup{V^\prime}$. Then
    \begin{enumerate}[(i)]
        \item if $i$, $j\in V$, then $\hat{T}_{ij}= 4 T_{ij}$;
        \item if $i\in V^\prime$, $j\in V$, $\hat{\Gamma}(i)=\{s,t\}$, then
            \begin{equation*}
                \begin{aligned}
                     \hat{T}_{ij}&= 1+2(T_{sj}+T_{tj});\\
                     \hat{T}_{ji}&= 2mq-1+2(T_{js}+T_{jt})-(T_{ts}+T_{st});
                \end{aligned}
            \end{equation*}
        \item if $i$, $j\in V^\prime$, $\hat{\Gamma}(i)=\{s,t\}$, $\hat{\Gamma}(j)=\{u,v\}$, then
            \begin{equation*}
                \hat{T}_{ij}= 2mq+T_{su}+T_{tu}+T_{sv}+T_{tv}-(T_{uv}+T_{vu}).
            \end{equation*}
    \end{enumerate}
%where $\hat{T}_{ij}$ denotes hitting times on graph $S_q(G)$, while $T_{ij}$ denotes $G$'s.
\end{theorem}

\begin{proof}
Note that $\hat{m}=2qm$, $\hat{d}_i=qd_i$ if $i\in V$, and $\hat{d}_i=2$ if $i\in V^\prime$.

We first prove (i).  We distinguish two cases:  (a)   $G$ is a non-bipartite graph, and   (b) $G$ is a  bipartite graph.  %We only prove the first case, since the second case is similar.
When $G$  is a non-bipartite graph,  by Theorems~\ref{Theorem1} and~\ref{Theorem2}, we have
\begin{equation*}
    \begin{aligned}
         \hat{T}_{ij} = &   2\hat{m}\Bigg(\sum_{k=2}^{n}
                                \bigg(\frac{1}{1-\sqrt{\frac{1+\lambda_k}{2}}}
                                +\frac{1}{1+\sqrt{\frac{1+\lambda_k}{2}}}\bigg)
                                \bigg(\frac{v_{kj}^2}{2qd_j}-\frac{v_{ki}v_{kj}}{2q\sqrt{d_id_j}}\bigg)\Bigg)\\
                          = &   8m\sum_{k=2}^n\frac{1}{1-\lambda_k}
                                \bigg(\frac{v_{kj}^2}{d_j}-\frac{v_{ki}v_{kj}}{\sqrt{d_id_j}}\bigg)
                          =     4T_{ij}.
    \end{aligned}
\end{equation*}
When $G$ is a non-bipartite graph, the proof is similar.  Thus (i) is proved.
%If $G$ is non-bipartite, then by Theorems~\ref{Theorem1},~\ref{Theorem2} and Eq.(\ref{re1}), we have
%\begin{equation*}
%    \begin{aligned}
%        \hat{T}_{ij}  = &   8m\sum_{k=2}^{n-1}\frac{1}{1-\lambda_k}
%                                \Big(\frac{v_{kj}^2}{d_j}-\frac{v_{ki}v_{kj}}{\sqrt{d_id_j}}\Big)
%                                +4mq\Big(\frac{v_{nj}^2}{qd_j}-\frac{v_{nj}v_{ni}}{q\sqrt{d_id_j}}\Big)\\
%                          = &   8m\sum_{k=2}^n\frac{1}{1-\lambda_k}
%                                \Big(\frac{v_{kj}^2}{d_j}-\frac{v_{ki}v_{kj}}{\sqrt{d_id_j}}\Big)
%                          =     4T_{ij}.
%    \end{aligned}
%\end{equation*}
%where the next-to-last equality is obtained by using the fact that $\lambda_n=-1$.

We continue to prove (ii). Since $\hat{\Gamma}(i)=\{s,t\}$,
\begin{equation*}
    \hat{T}_{ij}=1+\frac{1}{2}\left(\hat{T}_{sj}+\hat{T}_{tj}\right)=1+2(T_{sj}+T_{tj}).
\end{equation*}
While for $\hat{T}_{ji}$, we also divide it into two cases:  (a)   $G$  is a non-bipartite graph, and   (b) $G$ is a  bipartite graph.  For the first case that  $G$ is non-bipartite,  by Theorem~\ref{Theorem1} and Eqs.~(\ref{re1A}) and~(\ref{sum1}), we have
\begin{equation*}
    \begin{aligned}
        \hat{T}_{ji}=   &   1+2\hat{m}\Bigg(\sum_{k=2}^n
                                \bigg(\frac{1}{1-\sqrt{\frac{1+\lambda_k}{2}}}
                                +\frac{1}{1+\sqrt{\frac{1+\lambda_k}{2}}}\bigg)
                                \frac{1}{4q(1+\lambda_k)}\bigg(\frac{v_{ks}}{\sqrt{d_s}}+\frac{v_{kt}}{\sqrt{d_t}}\bigg)^2
                    \\      &   -\sum_{k=2}^n\bigg(\frac{1}{1-\sqrt{\frac{1+\lambda_k}{2}}}
                                -\frac{1}{1+\sqrt{\frac{1+\lambda_k}{2}}}\bigg)
                    \\      &   \times\frac{v_{kj}}{2q\sqrt{2(1+\lambda_k)d_j}}
                               \bigg (\frac{v_{ks}}{\sqrt{d_s}}+\frac{v_{kt}}{\sqrt{d_t}}\bigg)
                                +\sum_{z=1}^{mq-n}\frac{Y_{zi}^2}{2}\Bigg)
                    \\  =   &   1+4mq\Bigg(\sum_{k=2}^n\frac{1}{q(1-\lambda_k)(1+\lambda_k)}
                                \bigg(\frac{v_{ks}}{\sqrt{d_s}}+\frac{v_{kt}}{\sqrt{d_t}}\bigg)^2
                                -\sum_{k=2}^n\frac{1}{q(1-\lambda_k)}
                    \\      &   \bigg(\frac{v_{ks}v_{kj}}{\sqrt{d_sd_j}}+\frac{v_{kt}v_{kj}}{\sqrt{d_td_j}}\bigg)
                                +\frac{1}{2}\Bigg(1-\frac{1}{mq}
                                -\sum_{k=2}^n\frac{1}{q(1+\lambda_k)}
                                \bigg(\frac{v_{ks}}{\sqrt{d_s}}+\frac{v_{kt}}{\sqrt{d_t}}\bigg)^2\Bigg)\Bigg)
                    \\  =   &   2mq-1+4m\sum_{k=2}^n\frac{1}{1-\lambda_k}
                                \Bigg(\bigg(\frac{v_{ks}^2}{d_s}-\frac{v_{ks}v_{kj}}{\sqrt{d_sd_j}}\bigg)
                                +\bigg(\frac{v_{kt}^2}{d_t}-\frac{v_{kt}v_{kj}}{\sqrt{d_td_j}}\bigg)
                    \\      &   -\frac{1}{2}\bigg(\frac{v_{ks}}{\sqrt{d_s}}-\frac{v_{kt}}{\sqrt{d_t}}\bigg)^2\Bigg)
                                =2mq-1+2(T_{js}+T_{jt})-(T_{ts}+T_{st}).
    \end{aligned}
\end{equation*}
If $G$ is bipartite, our proof is similar. %We just need to notice that $\lambda_n=-1$.

We finally prove (iii). Considering $\hat{\Gamma}(i)=\{s,t\}$ and $\hat{\Gamma}(j)=\{u,v\}$,  we obtain
\begin{equation*}
    \begin{aligned}
        \hat{T}_{ij}    =   &   1+\frac{1}{2}\left(\hat{T}_{sj}+\hat{T}_{tj}\right)
                    \\      =   &   1+\frac{1}{2}\Big(2mq-1+2(T_{su}+T_{sv})-(T_{uv}+T_{vu})
                    \\          &   +2mq-1+2(T_{tu}+T_{tv})-(T_{uv}+T_{vu})\Big)
                    \\      =   &   2mq+T_{su}+T_{tu}+T_{sv}+T_{tv}-(T_{uv}+T_{vu}).
    \end{aligned}
\end{equation*}
This completes the proof.
\end{proof}

\subsection{Kemeny's Constant}

In addition to the two-node hitting time, the  Kemeny's constant of  $S_q(G)$ can also be expressed in terms of that of $G$.
\begin{theorem}\label{conKem}
Let $G$ be a simple connected graph with $n$ nodes and $m$ edges, and let $S_q(G)$ be the $q$-subdivision graph. Then
\begin{equation*}
K(S_q(G))=4K(G)+\frac{2mq-2n+1}{2}.
\end{equation*}
\end{theorem}
\begin{proof}
Suppose that $1=\lambda_1>\lambda_2\geq...\geq\lambda_n\geq-1$ are the eigenvalues of the matrix $P$. We first consider the case that $G$ is a non-bipartite graph. For this case, by Lemma~\ref{lemmaKem} and Theorem~\ref{Theorem2}, we have
\begin{equation*}
    \begin{aligned}
        K(S_q(G))   =&      \sum_{k=2}^n
                            \Bigg(\frac{1}{1-\sqrt{\frac{1+\lambda_k}{2}}}
                            +\frac{1}{1+\sqrt{\frac{1+\lambda_k}{2}}}\Bigg)
                            +\frac{1}{2}+mq-n\\
                    =&      4K(G)+\frac{2mq-2n+1}{2}.
    \end{aligned}
\end{equation*}
For the other case that $G$ is bipartite, we can prove similarly.
\end{proof}

\section{Resistance Distances of $q$-subdivision Graphs }

In this section, we determine the two-node resistance distance, multiplicative degree-Kirchhoff index, additive degree-Kirchhoff index, and  Kirchhoff index of  $S_q(G)$, in terms of those of  $G$.

\subsection{Two-Node Resistance Distance }

We first determine  the resistance distance between any pair of nodes in $S_q(G)$.
\begin{theorem}\label{con1}
Let $G$ be a simple connected graph with $n$ nodes and $m$ edges, and let $S_q(G)$ be the $q$-subdivision graph of  $G$ with node set $\hat{V}=V\cup{V^\prime}$. Then
\begin{enumerate}[(1)]
  \item for $i$, $j \in V$, $$\hat{r}_{ij}=\frac{2}{q}r_{ij};$$
  \item for $i\in V^\prime$, $j\in V$ and $\hat{\Gamma}(i)=\{s,t\}$,
    \begin{equation*}
    \hat{r}_{ij}=\frac{1}{2}+\frac{2r_{sj}+2r_{tj}-r_{st}}{2q};
    \end{equation*}
  \item for $i$, $j \in V^\prime$, $\hat{\Gamma}(i)=\{s,t\}$ and $\hat{\Gamma}(j)=\{u,v\}$,
    \begin{equation*}
    \hat{r}_{ij}=1+\frac{r_{su}+r_{tu}+r_{sv}+r_{tv}-r_{st}-r_{uv}}{2q}.
    \end{equation*}
\end{enumerate}
%where $\hat{r}_{ij}(res.~r_{ij})$ denotes the resistance distance of $S_q(G)(res.~G)$.
\end{theorem}
\begin{proof}
The results follow directly from Lemma~\ref{lemmaR} and Theorem~\ref{HT}.
\end{proof}

\subsection{Some Intermediary Results}

In the next  subsections, we will derive the Kirchhoff index, the additive degree-Kirchhoff index and the multiplicative degree-Kirchhoff index for $S_q(G)$. In the computation of the first two graph invariants, we need  the following two properties for resistance distance in $S_q(G)$.
\begin{lemma}\label{afore1}
Let $G$ be a simple connected graph with $n$ nodes and $m$ edges, and let $S_q(G)$ be the $q$-subdivision graph of  $G$ with node set $\hat{V}=V\cup{V^\prime}$. Then
\begin{equation*}
    \sum_{i\in V^\prime}\sum_{j\in V}\hat{r}_{ij}=
    \bar{\mathcal{K}}(G)+\frac{mnq-n^2+n}{2}.
\end{equation*}
\end{lemma}%The proof is in~\ref{app-1}.
\begin{proof}
Note that, $\sum_{i\in V^\prime}\sum_{j\in V}\hat{r}_{ij}$ can be divided into two sum terms as
\begin{equation}
    \sum_{i\in V^\prime}\sum_{j\in V}\hat{r}_{ij}=
        \sum_{i\in V^\prime} \sum_{j\in \hat{\Gamma}(i)} \hat{r}_{ij}
        +\sum_{i\in V^\prime} \sum_{j\in V\backslash\hat{\Gamma}(i)} \hat{r}_{ij}.
\end{equation}
We next compute the above two sum terms separately.

\begin{enumerate}[(i)]
\item
    As for the first term, by Lemma~\ref{Foster}, we have
    \begin{equation}\label{A1_1}
        \sum_{i\in V^\prime} \sum_{j\in \hat{\Gamma}(i)} \hat{r}_{ij}
        =\sum_{ij\in\hat{E}}\hat{r}_{ij}=|\hat{V}|-1=mq+n-1.
    \end{equation}

  \item
  As for the second term,  suppose that $\hat{\Gamma}(i)=\{s,t\}$. According to Eq.~(\ref{div1}), Lemma~\ref{Foster} and Theorem~\ref{con1}, we have
    \begin{equation}\label{A1_2}
        \begin{aligned}
            \sum_{i\in V^\prime} \sum_{j\in V\backslash\hat{\Gamma}(i)} \hat{r}_{ij}
            =&  \sum_{f=1}^{q}\sum_{i\in V^{(f)}} \sum_{j\in V\backslash\hat{\Gamma}(i)}
                \Big(\frac{1}{2}+\frac{2r_{sj}+2r_{tj}-r_{st}}{2q}\Big)\\
            =&  q\sum_{i\in V^{(1)}} \sum_{j\in V\backslash\hat{\Gamma}(i)}
                \Big(\frac{1}{2}+\frac{2r_{sj}+2r_{tj}-r_{st}}{2q}\Big)\\
            =&  \sum_{i\in V^{(1)}}\Bigg(
                \frac{(n-2)q}{2}
                +\sum_{j\in V\backslash\hat{\Gamma}(i)}\big(r_{sj}+r_{tj}\big)
                -\frac{n-2}{2}r_{st}\Bigg).
        \end{aligned}
    \end{equation}
   % where the second equation is obtained by the fact that $V^{(m)}$ behaves coherently.

    For convenience, let $r_s$ be the sum of resistance distances between $s$ and all other nodes in graph $G$, that is,
    \begin{equation*}
        r_{s}=\sum_{j\in V \atop j\neq s} r_{sj}.
    \end{equation*}
    Thus, Eq.~(\ref{A1_2}) can be rewritten as
    \begin{equation}\label{A1_3}
        \begin{aligned}
            \sum_{i\in V^\prime} \sum_{j\in V\backslash\hat{\Gamma}(i)} \hat{r}_{ij}
            =&  \sum_{i\in V^{(1)}}\Bigg(
                \frac{(n-2)q}{2}
                +r_s+r_t
                -\frac{n+2}{2}r_{st}\Bigg)\\
            =&  \frac{m(n-2)q}{2}+\sum_{i\in V^{(1)}}
               (r_s+r_t)
                -\frac{n+2}{2}\sum_{i\in V^{(1)}}r_{st}.
        \end{aligned}
    \end{equation}
The term $\sum_{i\in V^{(1)}}(r_s+r_t)$ can be computed as
     \begin{equation}\label{A1_4}
    \begin{aligned}
        \sum_{i\in V^{(1)}}(r_s+r_t)=&\sum_{st \in E}(r_s+r_t)
        =\sum_{s\in V}d_s r_s\\
        =&\sum_{\{i,j\}\subseteq V}(d_i+d_j)r_{ij}=\bar{\mathcal{K}}(G).
    \end{aligned}
    \end{equation}
   By Lemma~\ref{Foster}, the term $\frac{n+2}{2}\sum_{i\in V^{(1)}}r_{st}$ can be evaluated as
    \begin{equation}\label{A1_5}
        \frac{n+2}{2}\sum_{i\in V^{(1)}}r_{st}=
        \frac{n+2}{2}\sum_{st\in E}r_{st}=\frac{(n+2)(n-1)}{2}.
    \end{equation}

 Plugging Eqs.~(\ref{A1_4}) and~(\ref{A1_5}) into Eq.~(\ref{A1_3}) gives
    \begin{equation}\label{A1_6}
         \sum_{i\in V^\prime} \sum_{j\in V\backslash\hat{\Gamma}(i)} \hat{r}_{ij}
         = \frac{m(n-2)q}{2}+\bar{\mathcal{K}}(G)-\frac{(n+2)(n-1)}{2}.
    \end{equation}
\end{enumerate}
Combining Eqs.~(\ref{A1_1}) and (\ref{A1_6}) gives the desired result.
\end{proof}

\begin{lemma}\label{Afore2}
Let $G$ be a connected graph with $n$ nodes and $m$ edges, and let $S_q(G)$ be the $q$-subdivision graph of  $G$ with node set $\hat{V}=V\cup{V^\prime}$. Then
\begin{equation*}
    \sum_{\{i,j\}\subseteq V^\prime} \hat{r}_{ij}
    =\frac{q}{2}\tilde{\mathcal{K}}(G)+\frac{mq(mq-1)}{2}-\frac{m(n-1)q}{2}.
\end{equation*}
\end{lemma} %The proof is in~\ref{app-2}.
\begin{proof}
Suppose that $\hat{\Gamma}(i)=\{s,t\}$ and $\hat{\Gamma}(j)=\{u,v\}$. Then  by Theorem~\ref{con1},
\begin{equation}\label{A2_1}
  \begin{aligned}
        &   \sum_{\{i,j\}\subseteq V^\prime} \hat{r}_{ij}
            =\sum_{\{i,j\}\subseteq V^\prime}
           \left (1+\frac{r_{su}+r_{tu}+r_{sv}+r_{tv}}{2q}-\frac{r_{st}+r_{uv}}{2q}\right)\\
        =&  \frac{mq(mq-1)}{2}
            +\sum_{\{i,j\}\subseteq V^\prime}\frac{r_{su}+r_{tu}+r_{sv}+r_{tv}}{2q}
            -\sum_{\{i,j\}\subseteq V^\prime}\frac{r_{st}+r_{uv}}{2q}.
  \end{aligned}
\end{equation}

We now compute the second term in Eq.~\eqref{A2_1}. It is not easy to evaluate it directly. We will compute it in an alternative way.  For any pair of nodes $\{k,l\}\subseteq V$,  we consider how many times $r_{kl}$ appears in the summation. Observe that $r_{kl}$ is summed once if and only if there exists a unique subset $\{i,j\}\subseteq V^\prime$ such that $k\in \hat{\Gamma} (i)$ and $l\in\hat{\Gamma}(j)$. Thus, our problem could be simplified and converted to the following one: how many pairwise different aforementioned subsets exist?
It is not difficult to see that if $k$ is not adjacent to $l$ in $G$ there exist $q^2d_kd_l$ subsets, and that if $kl\in E(G)$, there exist $q^2d_kd_l-q$ such subsets. Then, once again by Lemma~\ref{Foster}, we have
\begin{equation}\label{A2_2}
  \begin{aligned}
        &   \sum_{\{i,j\}\subseteq V^\prime}\frac{r_{su}+r_{tu}+r_{sv}+r_{tv}}{2q}\\
        =&  \sum_{\{k,l\}\subseteq V \atop kl\notin E} \frac{q^2d_kd_l}{2q}r_{kl}
            +\sum_{\{k,l\}\subseteq V \atop kl\in E} \frac{q^2d_kd_l-q}{2q}r_{kl}\\
        =&  \sum_{\{k,l\}\subseteq V} \frac{qd_kd_l}{2}r_{kl}
            -\frac{1}{2}\sum_{\{k,l\}\subseteq V \atop kl\in E} r_{kl}
        =   \frac{q}{2}\sum_{\{k,l\}\subseteq V} d_kd_lr_{kl}-\frac{n-1}{2}\\
        =&  \frac{q}{2}\tilde{\mathcal{K}}(G)-\frac{n-1}{2}.
  \end{aligned}
\end{equation}

We proceed to evaluate the third term in Eq.~\eqref{A2_1}. Note that for any two different nodes $i$ and $j$ in $V^\prime$, if their neighbours are the same, i.e., $\hat{\Gamma}(i)=\hat{\Gamma}(j)=\{s,t\}$, we use  $i\sim j$ to denote this relation. Otherwise, the  sets  of  their neighbours are different, we call $i\nsim j$. According to these two relations and  Eq.~(\ref{div1}), it follows  that
\begin{equation}\label{A2_3}
  \begin{aligned}
        &   \sum_{\{i,j\}\subseteq V^\prime}\frac{r_{st}+r_{uv}}{2q}
        =   \frac{1}{4q}\sum_{i\in V^\prime}\sum_{j\in V^\prime}\big(r_{st}+r_{uv}\big)\\
        =&  \frac{1}{4q}\sum_{f=1}^{q}\sum_{i\in V^{(f)}}\bigg(
            \sum_{i\nsim j}(r_{st}+r_{uv})
            +\sum_{i\sim j}(r_{st}+r_{st})\bigg)\\
        =&  \frac{1}{4q} q\sum_{st\in E}\bigg(
            q\sum_{uv\in E \atop uv\neq st} (r_{st}+r_{uv})
            +2(q-1)r_{st}\bigg)\\
        =&  \frac{1}{4}\sum_{st\in E}\bigg(
            q\sum_{uv\in E}r_{uv}
            +(mq-2)r_{st}\bigg).
  \end{aligned}
\end{equation}
By Lemma~\ref{Foster}, Eq.~(\ref{A2_3}) can be recast as
\begin{equation}\label{A2_4}
    \begin{aligned}
            \sum_{\{i,j\}\subseteq V^\prime}\frac{r_{st}+r_{uv}}{2q}
        &=  \frac{1}{4}\sum_{st\in E}\Big((n-1)q+(mq-2)r_{st}\Big)\\
        &=  \frac{m(n-1)q}{4}+\frac{mq-2}{4}(n-1).
    \end{aligned}
\end{equation}
Plugging Eqs.~\eqref{A2_2} and~\eqref{A2_4} into Eq.~\eqref{A2_1} gives the result.
\end{proof}

\subsection{Multiplicative Degree-Kirchhoff Index}

We first  determine the multiplicative degree-Kirchhoff index for $S_q(G)$.
\begin{theorem}\label{con2}
Let $G$ be a connected graph with $n$ nodes and $m$ edges, and let $S_q(G)$ be the $q$-subdivision graph of  $G$. Then
\begin{equation*}
    \tilde{\mathcal{K}}(S_q(G))=8q\tilde{\mathcal{K}}(G)+2mq(2mq-2n+1).
\end{equation*}
\end{theorem}
\begin{proof}
According to Lemma~\ref{lemmaKem} and~\ref{lemmaK*}, Theorem~\ref{con2} is an obvious consequence of Theorem~\ref{conKem}.
\end{proof}

%In the next two paragraphs, we will dedicate to derive the additive degree-Kirchhoff index and the Kirchhoff index. Analogous to the aforementioned decimation method, we would scrutinize the close relationship between $S_q(G)$ and $G$. To complete our proof, several lemmas are needed and displayed bellow.

\subsection{Additive Degree-Kirchhoff Index}

We continue to determine the additive degree-Kirchhoff index for  $S_q(G)$.
\begin{theorem}\label{conK+}
Let $G$ be a simple connected graph with $n$ nodes and $m$ edges, and let $S_q(G)$ be the $ q$-subdivision graph of $G$. Then
\begin{equation*}
    \bar{\mathcal{K}}(S_q(G))=4\bar{\mathcal{K}}(G)+4q\tilde{\mathcal{K}}(G)+mq(3mq-2n+1)-n(n-1).
\end{equation*}
\end{theorem}
\begin{proof}
%First notice that $\hat{d}_i=qd_i$ if $i\in V$, while $\hat{d}_i=2$ if $i\in V^\prime$.
By  definition of the additive degree-Kirchhoff index, we have
\begin{equation}\label{Kf+1}
    \begin{aligned}
    \Bar{\mathcal{K}}(S_q(G))=&  \sum_{\{i,j\}\subseteq V\cup V^\prime} (\hat{d}_i+\hat{d}_j) \hat{r}_{ij}\\
                    =&  \sum_{\{i,j\}\subseteq V} (\hat{d}_i+\hat{d}_j) \hat{r}_{ij}
                        +\sum_{i\in V^\prime}\sum_{j\in V}(\hat{d}_i+\hat{d}_j) \hat{r}_{ij}
                        +\sum_{\{i,j\}\subseteq V^\prime} (\hat{d}_i+\hat{d}_j) \hat{r}_{ij}.
    \end{aligned}
\end{equation}
We now compute the three sum terms on the last row of Eq.~\eqref{Kf+1} one by one.

For the first sum term, by Theorem~\ref{con1}, we have
    \begin{equation}\label{Kf+2}
        \sum_{\{i,j\}\subseteq V} (\hat{d}_i+\hat{d}_j) \hat{r}_{ij}
        =\sum_{\{i,j\}\subseteq V} (qd_i+qd_j)  \frac{2}{q}r_{ij}
        =2\Bar{\mathcal{K}}(G).
    \end{equation}

For the second sum term, it can be evaluated as
  \begin{equation}\label{Kf+3}
    \begin{aligned}
        \sum_{i\in V^\prime}\sum_{j\in V}(\hat{d}_i+\hat{d}_j) \hat{r}_{ij}
        =&  \sum_{i\in V^\prime}\sum_{j\in V}(2+qd_j) \hat{r}_{ij}\\
        =&  2\sum_{i\in V^\prime}\sum_{j\in V}\hat{r}_{ij}
            +q\sum_{i\in V^\prime}\sum_{j\in V}d_j\hat{r}_{ij}.
    \end{aligned}
  \end{equation}
By Lemma~\ref{afore1}, we  have
  \begin{equation}\label{Kf+4}
    2\sum_{i\in V^\prime}\sum_{j\in V}\hat{r}_{ij}=
    2\bar{\mathcal{K}}(G)+mnq-n^2+n.
  \end{equation}
On the other hand, by Lemma~\ref{Foster} and Theorem~\ref{con1},
  \begin{equation}\label{Kf+5}
    \begin{aligned}
          q\sum_{i\in V^\prime}\sum_{j\in V}d_j\hat{r}_{ij}
          =&    \sum_{i\in V^\prime}\sum_{j\in V}qd_j
                \Big(\frac{1}{2}+\frac{2r_{sj}+2r_{tj}-r_{st}}{2q}\Big)\\
          =&    \frac{q}{2}\sum_{i\in V^\prime}\sum_{j\in V}d_j
                +\sum_{i\in V^\prime}\sum_{j\in V}d_j(r_{sj}+r_{tj})
                -\frac{1}{2}\sum_{i\in V^\prime}\sum_{j\in V}d_jr_{st}\\
          =&    \frac{q}{2}\sum_{i\in V^\prime}2m
                +\sum_{i\in V^\prime}\sum_{j\in V}d_j(r_{sj}+r_{tj})
                -\frac{1}{2}\sum_{i\in V^\prime}2mr_{st}\\
          =&    m^2q^2+\sum_{i\in V^\prime}\sum_{j\in V}d_j(r_{sj}+r_{tj})
                -mq(n-1).
    \end{aligned}
  \end{equation}
%  where the last equality derived for $V^{(m)}$ behaves coherently. Analogously, for the middle part of Eq.(\ref{Kf+5}),
 For the middle term in Eq.~\eqref{Kf+5}, we have
  \begin{equation}\label{Kf+6}
    \begin{aligned}
        \sum_{i\in V^\prime}\sum_{j\in V}d_j(r_{sj}+r_{tj})
        =&  q\sum_{i\in V^{(1)}}\sum_{j\in V}d_j(r_{sj}+r_{tj})
        =   q\sum_{j\in V}\sum_{i\in V^{(1)}}d_j(r_{sj}+r_{tj})\\
        =&  q\sum_{j\in V}\sum_{k\in V}d_j d_k r_{kj}
        =   2q\tilde{\mathcal{K}}(G).
    \end{aligned}
  \end{equation}
Combining  Eqs.~(\ref{Kf+3})-(\ref{Kf+6}) yields
  \begin{equation}\label{Kf+7}
    \sum_{i\in V^\prime}\sum_{j\in V}(\hat{d}_i+\hat{d}_j) \hat{r}_{ij}
    =2\Bar{\mathcal{K}}(G)+2q\tilde{\mathcal{K}}(G)+m^2q^2+mq-n^2+n.
  \end{equation}

For the third sum term in Eq.~(\ref{Kf+1}), by Lemma~\ref{Afore2}, we have
    \begin{equation}\label{Kf+8}
        \sum_{\{i,j\}\subseteq V^\prime} (\hat{d}_i+\hat{d}_j) \hat{r}_{ij}
        =4\sum_{\{i,j\}\subseteq V^\prime} \hat{r}_{ij}
        =2q\tilde{\mathcal{K}}(G)+2m^2q^2-2mnq.
    \end{equation}

Substituting Eqs.~(\ref{Kf+2}) (\ref{Kf+7}) and (\ref{Kf+8}) back into Eq.~(\ref{Kf+1}),  our proof is completed  after simple calculations.
\end{proof}

\subsection{Kirchhoff Index}

%With the above-obtained intermediary result,
We finally  determine the Kirchhoff index for $S_q(G)$.
\begin{theorem}\label{conK}
Let $G$ be a simple connected graph with $n$ nodes and $m$ edges, and let $S_q(G)$ be the $ q$-subdivision graph of $G$. Then
\begin{equation*}
    \mathcal{K}(S_q(G))=\frac{2}{q}\mathcal{K}(G)+\bar{\mathcal{K}}(G)+\frac{q}{2}\tilde{\mathcal{K}}(G)
    +\frac{m^2q^2-n(n-1)}{2}.
\end{equation*}
\end{theorem}

\begin{proof}
According to Definition~\ref{defK} and Eq.~(\ref{div1}), we have
    \begin{equation}\label{Kf1}
    \begin{aligned}
        \mathcal{K}(S_q(G))  =&  \sum_{\{i,j\}\subseteq\hat{V}} \hat{r}_{ij}
                        =   \sum_{\{i,j\}\subseteq V\bigcup V^\prime} \hat{r}_{ij}\\
                        =&  \sum_{\{i,j\}\subseteq V} \hat{r}_{ij}
                            +\sum_{i\in V^\prime} \sum_{j\in V} \hat{r}_{ij}
                            +\sum_{\{i,j\}\subseteq V^\prime} \hat{r}_{ij}.
    \end{aligned}
    \end{equation}
%In order to complete our proof more clearly,
We shall compute the three sum terms in Eq.~\eqref{Kf1} one by one.
%\begin{enumerate}[(i)]

For the first sum term, by Theorem~\ref{con1},
    \begin{equation}\label{Kf2}
        \sum_{\{i,j\}\subseteq V} \hat{r}_{ij}=\sum_{\{i,j\}\subseteq V} \frac{2}{q} r_{ij}
        =\frac{2}{q}\mathcal{K}(G).
    \end{equation}
 For the second sum term, by Lemma~\ref{afore1} we obtain
    \begin{equation}\label{Kf3}
        \sum_{i\in V^\prime}\sum_{j\in V}\hat{r}_{ij}=
        \bar{\mathcal{K}}(G)+\frac{mnq-n^2+n}{2}.
    \end{equation}
For the third sum term, by Lemma~\ref{Afore2}, we have
    \begin{equation}\label{Kf4}
    \sum_{\{i,j\}\subseteq V^\prime} \hat{r}_{ij}
    =\frac{q}{2}\tilde{\mathcal{K}}(G)+\frac{mq(mq-1)}{2}-\frac{m(n-1)q}{2}.
    \end{equation}
%\end{enumerate}

Plugging Eqs.~(\ref{Kf2})-(\ref{Kf4})  into Eq.~(\ref{Kf1}) leads to the desired result.
\end{proof}

\section{Properties  of Iterated $q$-subdivision graphs and Their Applications}

The $q$-subdivision graphs have found many applications in physics and network science. For example, by iteratively applying   $q$-subdivision operation on an edge we can obtain the hierarchical lattices, which can be used to mimic complex networks with the striking scale-free fractal topologies~\cite{ZhZhZo07}. In this section, we study the properties of iterated $q$-subdivision graphs, based on which we further obtain exact expressions for  some interesting quantities  for  the hierarchical lattices.

\subsection{Definition  of Iterated $q$-subdivision Graphs}

The family of iterated $q$-subdivision graphs $S_{q,k}(G)$ of a graph $G$ is defined as follows. For $k=0$, $S_{q,0}(G)=G$. For $k\geq 1$, $S_{q,k}(G)$ is obtained from $S_{q, k-1}(G)$ by performing the $q$-subdivision operation on $S_{q, k-1}(G)$. In other words, $S_{q,k}(G)=S_q(S_{q,k-1}(G))$.  % Through iterated constructions, we could derive a family of graphs. Because of the self-similarity properties shared among the whole family, the iterated $q$-th subdivision graphs is of vital significance. In this section, we shall study important invariants mentioned above of iterated $q$-subdivisions, which could be derived from our previous work easily.
For a quantity $Z$ of $G$, we use  $Z_{q,k}$  to denote the corresponding quantity associated with $S_{q,k}(G)$.  Then, in $S_{q,k}(G)$, the number of edges is
\begin{equation}\label{iterated1A}
         m_{q,k} =  2q m_{q,k-1}=(2q)^{k}m,
\end{equation}
and the number of   nodes is
\begin{equation}\label{iterated1B}
        n_{q,k} =  n_{q,k-1}+q m_{q,k-1} =\frac{mq\big[(2q)^{k}-1\big]}{2q-1}+n.
\end{equation}

\subsection{Formulas of  Quantities   of Iterated $q$-subdivision Graphs}

We here present expressions for some interesting quantities  for iterated $q$-subdivision graphs $S_{q,k}(G)$. % including  for are ready to give the formulae for these graph invariants of $S_{q,k}(G)$.

\subsubsection{Kemeny's Constant}
\begin{theorem}\label{iteratedKem}
Let $G$ be a simple connected graph with $n$ nodes and $m$ edges. Then
\begin{enumerate}[(i)]
  \item if $q=2$,
  \begin{equation*}
    K_{2,k} =4^kK_{2,0}
            +\frac{mk4^k}{3}
            +\Big(\frac{4m+3}{6}-n\Big)\frac{4^k-1}{3};
  \end{equation*}
  \item if $q\neq2$,
  \begin{equation*}
    \begin{aligned}
    K_{q,k} =&4^kK_{q,0}
            +\frac{mq(q-1)}{(q-2)(2q-1)}\left[(2q)^k-4^k\right]\\
            &+\Big(\frac{2mq+2q-1}{2(2q-1)}-n\Big)\frac{4^k-1}{3}.
    \end{aligned}
  \end{equation*}
\end{enumerate}
\end{theorem}
\begin{proof}
According to Theorem~\ref{conKem} and Eqs.~(\ref{iterated1A}) and~(\ref{iterated1B}) , we obtain
\begin{equation*}
    \begin{aligned}
        K_{q,k} =&  4K_{q,k-1}+\frac{1}{2}\left(2m_{q,k-1}q-2n_{q,k-1}+1\right)\\
                =&  4K_{q,k-1}+mq(2q)^{k-1}-\frac{mq\big[(2q)^{k-1}-1\big]}{2q-1}-n+\frac{1}{2}.
    \end{aligned}
\end{equation*}
Divided by $4^k$ on both sides, we obtain
\begin{equation*}
    \frac{K_{q,k}}{4^k}-\frac{K_{q,k-1}}{4^{k-1}}
    =\frac{m(q-1)}{2q-1}\Big(\frac{q}{2}\Big)^{k}
    +\Big(\frac{2mq+2q-1}{2(2q-1)}-n\Big)\frac{1}{4^k}.
\end{equation*}
If $q\neq2$, we  derive the following relation
\begin{equation*}
    \frac{K_{q,k}}{4^k}-\frac{K_{q,0}}{4^0}
    =\frac{mq(q-1)}{2(2q-1)}\frac{1-(q/2)^k}{1-q/2}
    +\Big(\frac{2mq+2q-1}{2(2q-1)}-n\Big)\frac{1/4\big[1-(1/4)^k\big]}{1-1/4},
\end{equation*}
which leads to the result through simple calculations.

For the case $q=2$, the proof is similar. %but a little easier
\end{proof}

\subsubsection{Multiplicative Degree-Kirchhoff Index}

\begin{theorem}\label{iteratedK*}
Let $G$ be a simple connected graph with $n$ nodes and $m$ edges. Then
\begin{enumerate}[(i)]
  \item if $q=2$,
    \begin{equation*}
    \tilde{\mathcal{K}}_{2,k}    =   16^{k}\tilde{\mathcal{K}}_{2,0}
                        +\frac{2m^2k16^{k}}{3}
                        +\Big(\frac{4m^2+3m}{3}-2mn\Big)\frac{16^{k}-4^{k}}{3};
    \end{equation*}
  \item if $q\neq2$,
    \begin{equation*}
    \begin{aligned}
    \tilde{\mathcal{K}}_{q,k}    =&  (8q)^k\tilde{\mathcal{K}}_{q,0}
                        +\frac{2m^2q(q-1)}{(q-2)(2q-1)}\Big[(2q)^{2k}-(8q)^k\Big]\\
                        &+\Big(\frac{2m^2q+2mq-m}{2q-1}-2mn\Big)\frac{(8q)^k-(2q)^k}{3}.
    \end{aligned}
    \end{equation*}
\end{enumerate}
\end{theorem}
\begin{proof}
By Lemmas~\ref{lemmaKem} and~\ref{lemmaK*},  the result follows directly from Theorem~\ref{iteratedKem}.
\end{proof}

\subsubsection{Additive Degree-Kirchhoff Index}
\begin{theorem}\label{iteratedK+}
Let $G$ be a simple connected graph with $n$ nodes and $m$ edges. Then
\begin{enumerate}[(i)]
  \item if $q=2$,
  \begin{equation*}
    \begin{aligned}
    \bar{\mathcal{K}}_{2,k}     =&  4^k\bar{\mathcal{K}}_{2,0}
                                    +\frac{2\left(16^k-4^k\right)\tilde{\mathcal{K}}_{2,0}}{3}
                                    +\frac{16^k-4^k}{9}2m(2m-2n+1)\\
                                &   +\frac{16^k}{9}4m^2k
                                    -\frac{4^k-1}{27}(2m-3n)(2m-3n+3);
    \end{aligned}
  \end{equation*}
  \item if $q\neq2$,
  \begin{equation*}
    \begin{aligned}
    \bar{\mathcal{K}}_{q,k} =&  4^k\bar{\mathcal{K}}_{q,0}
                                +\frac{3m^2q^3(q-1)\big[(2q)^{2k}-4^k\big]}{(q-2)(q+1)(2q-1)^2}\\
                            &   +\frac{\big[(8q)^k-4^k\big]}{2q-1}\Big(q\tilde{\mathcal{K}}_{q,0}
                                -\frac{2m^2q^2}{3(q-2)}-\frac{m(2n-1)q}{3}\Big)\\
                            &   -\frac{mq\big[(2q)^k-4^k\big]}{3(2q-1)}
                                \Big(\frac{2mq}{2q-1}
                                -2n+1\Big)\\
                            &   -\frac{4^k-1}{3}
                                \bigg(\Big(\frac{mq}{2q-1}-n\Big)\Big(\frac{mq}{2q-1}-n+1\Big)\bigg).
    \end{aligned}
  \end{equation*}
\end{enumerate}
\end{theorem}
\begin{proof}
By Theorem~\ref{conK+} and Eqs.~(\ref{iterated1A}) and~(\ref{iterated1B}), we have
\begin{equation}\label{add}
\begin{aligned}
    \bar{\mathcal{K}}_{q,k} =&  4\bar{\mathcal{K}}_{q,k-1}+4q\tilde{\mathcal{K}}_{q,k-1}
                                +m_{q,k-1}q\big(3m_{q,k-1}q-2n_{q,k-1}+1\big)\\
                            &   -n_{q,k-1}\big(n_{q,k-1}-1\big)\\
                            =&  4\bar{\mathcal{K}}_{q,k-1}+4q\tilde{\mathcal{K}}_{q,k-1}
                                +(2q)^{2k-2}\frac{4m^2q^2(3q^2-4q+1)}{(2q-1)^2}\\
                            &   +(2q)^{k-1}\Big(\frac{4m^2q^3}{(2q-1)^2}
                                -\frac{2m(2n-1)q^2}{2q-1}\Big)\\
                            &   -\Big(n-\frac{mq}{2q-1}\Big)^2+n-\frac{mq}{2q-1}.
\end{aligned}
\end{equation}
For $q=2$, inserting Theorem~\ref{iteratedK*} into Eq.~(\ref{add}) yields
\begin{equation}
\begin{aligned}
    \bar{\mathcal{K}}_{2,k} =&  4\bar{\mathcal{K}}_{2,k-1}+8\tilde{\mathcal{K}}_{2,k-1}
                                +\frac{80m^2\cdot16^{k-1}}{9}                               +4^{k-1}\Big(\frac{32m^2}{9}-\frac{8m(2n-1)}{3}\Big)\\
                            &   -\frac{(2m-3n)(2m-3n+3)}{9}\\
                            =&  4\bar{\mathcal{K}}_{2,k-1}
                                +16^{k-1}\left(8\tilde{\mathcal{K}}_{2,0}
                                +\frac{8m\left((6k+8)m-6n+3\right)}{9}\right)\\
                            &   -\frac{(2m-3n)(2m-3n+3)}{9}.
\end{aligned}
\end{equation}
By dividing both sides by $4^k$, we could derive the result through simple calculations.

Analogously, if $q\neq2$, using Theorem~\ref{iteratedK*}, we rewrite Eq.~(\ref{add}) as
\begin{equation}
\begin{aligned}
    \Bar{\mathcal{K}}_{q,k}=&  4\bar{\mathcal{K}}_{q,k-1}+
                    (2q)^{2k-2}\frac{12m^2q^3(q-1)^2}{(q-2)(2q-1)^2}
                    +(8q)^{k-1}\Big(4q\tilde{\mathcal{K}}_{q,0}-\frac{8m^2q^2}{3(q-2)}\\
                 &  -\frac{4m(2n-1)q}{3}\Big)
                    -(2q)^{k-1}\Big(\frac{4m^2q^2(q-2)}{3(2q-1)^2}
                    -\frac{2m(2n-1)q(q-2)}{3(2q-1)}\Big)\\
                 &  -\Big(n-\frac{mq}{2q-1}\Big)^2+n-\frac{mq}{2q-1}.
\end{aligned}
\end{equation}
Once again, by dividing both sides by $4^k$, we obtain a geometric sequence, which is solved to yield the result.
\end{proof}

\subsubsection{Kirchhoff Index}

\begin{theorem}\label{iteratedK}
Let $G$ be a simple connected graph with $n$ nodes and $m$ edges. Then
\begin{enumerate}[(1)]
  \item if $q=2$,
    \begin{equation*}
    \begin{aligned}
        \mathcal{K}_{2,k}  =&   \mathcal{K}_{2,0}
                                +\frac{4^k-1}{3}\bar{\mathcal{K}}_{2,0}
                                +\frac{(4^k-1)^2}{9}\tilde{\mathcal{K}}_{2,0}
                                +\frac{16^k-1}{135}m\left((10k+14)m-10n+5\right)\\
                            &   +\frac{4^k-1}{81}\left(-16m^2+24mn-12m-9n(n-1)\right)\\
                            &   -\frac{k(2m-3n)(2m-3n+3)}{54};
    \end{aligned}
  \end{equation*}
  \item if $q\neq2$,
  \begin{equation*}
    \begin{aligned}
        \mathcal{K}_{q,k}=&  \Big(\frac{2}{q}\Big)^k\mathcal{K}_{q,0}
                        +\frac{q[4^k-(2/q)^k]}{2(2q-1)}\bar{\mathcal{K}}_{q,0}
                        +\frac{q^2\big[(8q)^{k}-2\cdot4^k+(2/q)^k\big]}{4(2q-1)^2}\tilde{\mathcal{K}}_{q,0}\\
                     &  +\frac{m^2q^3(q-1)\big[(2q)^{2k}-(2/q)^k\big]}{2(q-2)(q+1)(2q-1)^2}
                        -\frac{mq^2\big[(8q)^{k}-(2/q)^k\big]}{6(2q-1)^2}
                        \Big(\frac{mq}{q-2}+n-\frac{1}{2}\Big)\\
                     &  +\frac{mq^2\big[(2q)^{k}-(2/q)^k\big]}{3(q+1)(2q-1)}
                        \Big(-\frac{mq}{2q-1}+n-\frac{1}{2}\Big)
                        -\frac{q\big[4^k-(2/q)^k\big]}{2q-1}\\
                     &  \Bigg(\frac{m^2q^2(q-1)}{2(2q-1)^2(q+1)}
                        +\frac{1}{6}\Big(\frac{mq}{2q-1}-n\Big)\Big(\frac{mq}{2q-1}-n+1\Big)\Bigg)\\
                     &  +\frac{q\big[(2/q)^k-1\big]}{6(q-2)}
                        \Big(\frac{mq}{2q-1}-n\Big)\Big(\frac{mq}{2q-1}-n+1\Big).
    \end{aligned}
  \end{equation*}
\end{enumerate}
\end{theorem}
\begin{proof}
By Theorem~\ref{conK}, we have
\begin{equation}\label{kir}
\begin{aligned}
    \mathcal{K}_{q,k}  =&   \frac{2}{q}\mathcal{K}_{q,k-1}+\bar{\mathcal{K}}_{q,k-1}+\frac{q}{2}\tilde{\mathcal{K}}_{q,k-1}
                            +\frac{m_{q,k-1}^2 q^2-n_{q,k-1}(n_{q,k-1}-1)}{2}\\
                        =&  \frac{2}{q}\mathcal{K}_{q,k-1}+\bar{\mathcal{K}}_{q,k-1}+\frac{q}{2}\tilde{\mathcal{K}}_{q,k-1}
                            +(2q)^{2k}\frac{m^2q(q-1)}{2(2q-1)^2}\\
                        &   -\frac{(2q)^km}{4(2q-1)}\Big(2n-\frac{2mq}{2q-1}-1\Big)
                            -\frac{1}{2}\Big(n-\frac{mq}{2q-1}\Big)\Big(n-\frac{mq}{2q-1}-1\Big).
\end{aligned}
\end{equation}
We first consider the case $q=2$. Inserting Theorems~\ref{iteratedK*} and~\ref{iteratedK+} into Eq.~(\ref{kir}) yields
\begin{equation}
\begin{aligned}
    \mathcal{K}_{2,k}   =&  \mathcal{K}_{2,k-1}+\bar{\mathcal{K}}_{2,k-1}+\tilde{\mathcal{K}}_{2,k-1}
                            +\frac{16^k}{9} m^2
                            +\frac{4^k}{36}m(4m-6n+3)\\
                        &   -\frac{(2m-3n)(2m-3n+3)}{18}\\
                        =&  \mathcal{K}_{2,k-1}+4^{k-1}\bar{\mathcal{K}}_{2,0}
                            +\frac{5\cdot16^{k-1}-2\cdot4^{k-1}}{3}\tilde{\mathcal{K}}_{2,0}\\
                        &   +\frac{16^{k-1}}{9}m
                            \bigg((10k+14)m-10n+5\bigg)\\
                        &   +\frac{4^{k-1}}{27}\bigg(-16m^2+24mn-12m-9n(n-1)\bigg)\\
                        &   -\frac{(2m-3n)(2m-3n+3)}{54},
\end{aligned}
\end{equation}
which could lead to our result through simple calculations.

For the other case $q\neq2$, once again by Theorems~\ref{iteratedK*} and~\ref{iteratedK+}, Eq.~(\ref{kir}) is rewritten as
\begin{equation}
\begin{aligned}
    \mathcal{K}_{q,k}  =&  \frac{2}{q}\mathcal{K}_{q,k-1}
                    +(2q)^{2k}\frac{m^2(q-1)(2q^3-1)}{4(q-2)(q+1)(2q-1)^2}
                    +(8q)^k \Big(\frac{2q+1}{16(2q-1)}\tilde{\mathcal{K}}_{q,0}\\
                &   -\frac{m^2q(2q+1)}{24(q-2)(2q-1)}
                    -\frac{m(2n-1)(2q+1)}{48(2q-1)}\Big)\\
                &   +(2q)^k\Big(\frac{m(2n-1)(q-1)}{6(2q-1)}-\frac{m^2q(q-1)}{3(2q-1)^2}\Big)
                    +4^k\Big(\frac{\bar{\mathcal{K}}_{q,0}}{4}
                    -\frac{q\tilde{\mathcal{K}}_{q,0}}{4(2q-1)}\\
                &   -\frac{m^2q^2(q-1)}{4(2q-1)^2(q+1)}
                    -\frac{1}{12}(n-\frac{mq}{2q-1})(n-\frac{mq}{2q-1}-1)\Big)\\
                &   -\frac{1}{6}(n-\frac{mq}{2q-1})(n-\frac{mq}{2q-1}-1).
\end{aligned}
\end{equation}
Dividing both sides by $(\frac{q}{2})^k$, we obtain a  geometric sequence, which is summed to yield the result.

For $q=2$, we could derive the result similarly.
\end{proof}

Our results in this section generalize those previously obtained for subdivision graphs~\cite{Ya14,YaKl15}, but our computation method is much simpler.

\subsection{Applications to the Hierarchical Lattices}

The hierarchical lattices~\cite{Ya88} are a particular example of  iterated $q$-subdivision graphs. They are constructed in an iterative way. Let $H_{q,k}$, $q\geq2$ and $k\geq 0$, denote the hierarchical lattices after $k$ iterations. For $k=0$, $H_{q,0}$ is an edge connecting two nodes. For $k\geq1$, $H_{q,k}$ is obtained from $H_{q,k-1}$ by performing  the $q$-subdivision operation on $H_{q,k-1}$. Thus, the hierarchical lattices are actually iterated $q$-subdivision graphs  $S_{q,k}(G)$ when $G$ is a graph consisting of two nodes linked by an edge. They  have been recently introduced as a model of complex networks with  scale-free fractal properties~\cite{ZhZhZo07}.  Fig.~\ref{Flower22} illustrates a particular   hierarchical lattice $H_{2,5}$. Below, we present some properties of the hierarchical lattices, by using the results  derived in last subsections.

%%%%%%%%%%%%%%%%%%%%%%%%%%%%%%%%%%%%%%%%%%%%%%%%%%%%%%%%
%Figure
%%%%%%%%%%%%%%%%%%%%%%%%%%%%%%%%%%%%%%%%%%%%%%%%%%%%%%%%%
\begin{figure}
\begin{center}
\includegraphics[width=7.0cm]{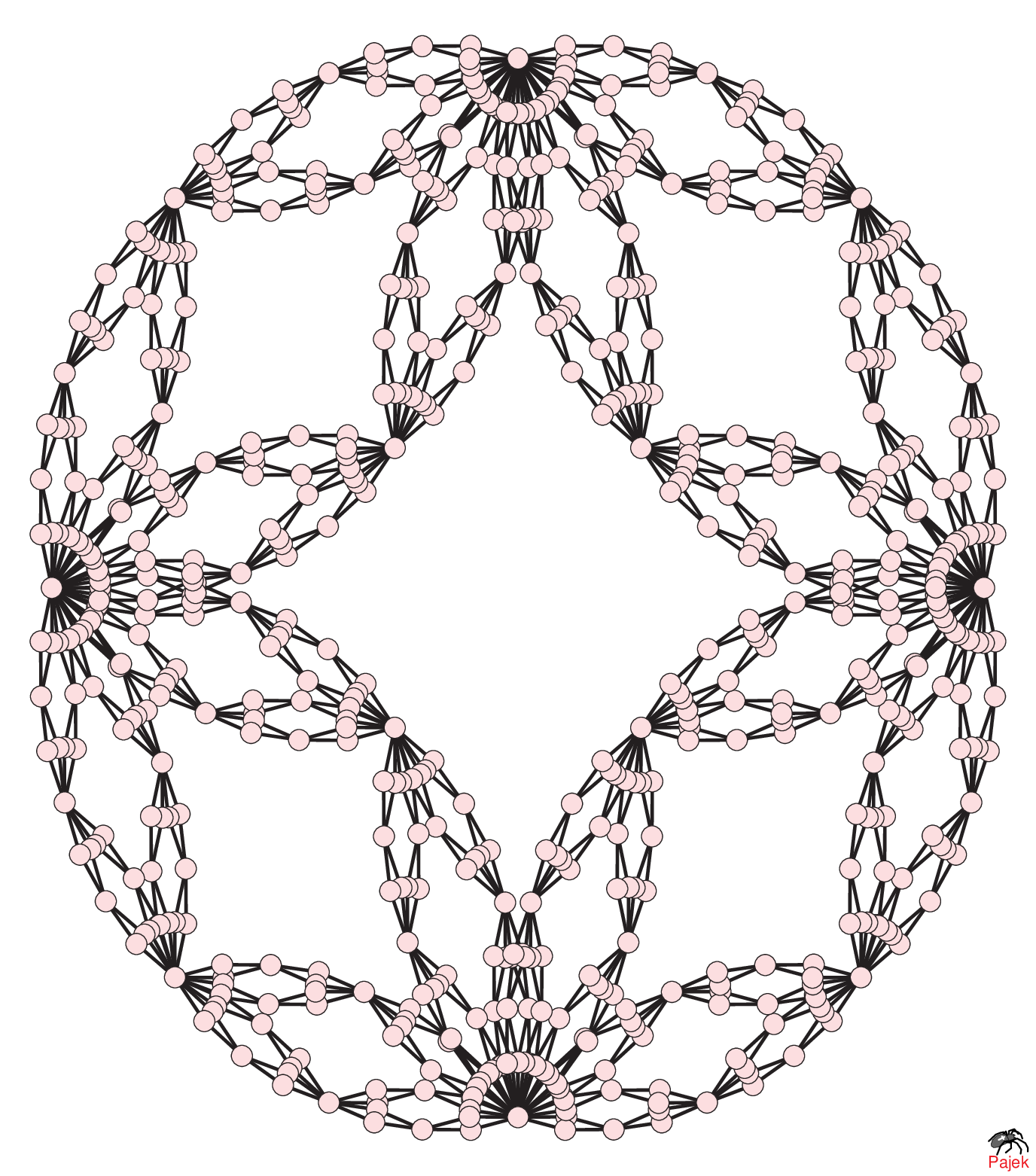}
\caption{The hierarchical lattice $H_{2,5}$.} \label{Flower22}
\end{center}
\end{figure}
%%%%%%%%%%%%%%%%%%%%%%%%%%%%%%%%%%%%%%%%%%%%%%%%%%%%%%%%%

%The hierarchical lattices \cite{yang1988family} are constructed in an iterative manner. We denote the hierarchical lattices(networks) after $k$ generations by $H_{q,k}$, $q\geq2$ and $k\geq0$. The networks are constructed as follows: for $k=0$, $H_{q,0}$ is an edge connecting two points. For $k\geq1$, $H_{q,k}$ is obtained from $H_{q,k}$ by the $q$-th subdivision operation. In this section, we present some properties of the hierarchical lattices, which could be derived from our work directly.

For $H_{q,0}$, its adjacency matrix and normalized  adjacency matrix  are both
$\left(
          \begin{array}{cc}
            0 & 1 \\
            1 & 0 \\
          \end{array}
        \right)$. Their eigenvalues  are $\lambda_1=1$ and $\lambda_2=-1$, with corresponding orthonormal eigenvectors being
 $\left(\frac{1}{\sqrt{2}}, \frac{1}{\sqrt{2}}\right)$ and     $\left(\frac{1}{\sqrt{2}}, -\frac{1}{\sqrt{2}}\right)$.
In addition, for $H_{q,0}$ the Kemeny constant, multiplicative degree-Kirchhoff index, additive degree-Kirchhoff index, and Kirchhoff index are  $K(H_{q,0})=\frac{1}{2}$, $\tilde{\mathcal{K}}(H_{q,0})=1$,  $\bar{\mathcal{K}}(H_{q,0})=2$, and $ \mathcal{K}(H_{q,0})=1$, respectively. Then, by Theorems~\ref{iteratedKem},~\ref{iteratedK*},~\ref{iteratedK+}, and~\ref{iteratedK}, we obtain the following exact solutions to the Kemeny constant $K(H_{q,k})$, multiplicative degree-Kirchhoff index $\tilde{\mathcal{K}}(H_{q,k})$, additive degree-Kirchhoff index $\bar{\mathcal{K}}(H_{q,k}) $, and Kirchhoff index $\mathcal{K}(H_{q,k})$ for  $H_{q,k}$.
\begin{equation}\label{egKem}
    K(H_{q,k}) =
    \left\{
        \begin{array}{ll}
            \frac{(6k+4)4^k+5}{18}, & \hbox{if $q=2$;} \\
            \frac{q(q-1)(2q)^k}{(q-2)(2q-1)}-\frac{q\cdot4^k}{3(q-2)}+\frac{4q-3}{6(2q-1)},
            & \hbox{if $q\neq2$.}
        \end{array}
    \right.
\end{equation}
\begin{equation}\label{egK*}
    \tilde{\mathcal{K}}(H_{q,k}) =
    \left\{
      \begin{array}{ll}
        \frac{(6k+4)16^{k}+5\cdot4^k}{9},
        & \hbox{if $q=2$;} \\
        \frac{2q(q-1)(2q)^{2k}}{(q-2)(2q-1)}-\frac{2q(8q)^k}{3(q-2)}
        +\frac{(4q-3)(2q)^k}{3(2q-1)},
        & \hbox{if $q\neq2$.}
      \end{array}
    \right.
\end{equation}
\begin{equation}\label{K+}
    \bar{\mathcal{K}}(H_{q,k}) =
    \left\{
      \begin{array}{ll}
        \frac{4(k+1)16^k}{9}+\frac{38\cdot4^k}{27}+\frac{4}{27},
        & \hbox{if $q=2$;} \\
        \frac{3q^3(q-1)(2q)^{2k}}{(q-2)(q+1)(2q-1)^2}
        -\frac{2q^2(8q)^k}{3(q-2)(2q-1)}
        +\frac{q(4q-3)(2q)^k}{3(2q-1)^2}\\
        +\frac{2(3q^2+2q-2)4^k}{3(q+1)(2q-1)}
        +\frac{(q-1)(3q-2)}{3(2q-1)^2},
        & \hbox{if $q\neq2$.}
      \end{array}
    \right.
\end{equation}
\begin{equation}\label{egK}
    \mathcal{K}(H_{q,k}) =
    \left\{
      \begin{array}{ll}
        \frac{2(5k+7)16^k}{135}+\frac{38\cdot4^k}{81}
        +\frac{173-60k}{405},
        & \hbox{if $q=2$;} \\
        \frac{q^3(q-1)(2q)^{2k}}{2(q-2)(q+1)(2q-1)^2}
        -\frac{q^3(8q)^k}{6(q-2)(2q-1)^2}
        +\frac{q^2(4q-3)(2q)^k}{6(q+1)(2q-1)^2}\\
        +\frac{q(3q^2+2q-2)4^k}{3(q+1)(2q-1)^2}
        +\frac{5q^4-9q^3-5q^2+12q-4}{2(q-2)(q+1)(2q-1)^2}\Big(\frac{2}{q}\Big)^k
        -\frac{(q-1)q(3q-2)}{6(q-2)(2q-1)^2},
        & \hbox{if $q\neq2$.}
      \end{array}
    \right.
\end{equation}

We note that Eq.~\eqref{egKem} is in complete agreement with the result obtained in \cite{ZhShHuCh12}.

\section{Conclusions}

The $q$-subdivision operation is an extension of traditional subdivision operation on a graph, which has been applied to construct complex networks. In this paper, we studied various properties of $q$-subdivision graph $S_q(G)$ of a simple connected graph $G$, and expressed some quantities of $S_q(G)$ in terms of associated with $G$. We first derived formulas for eigenvalues and eigenvectors of normalized adjacency matrix for $S_q(G)$. We then determined two-node hitting time and resistance distance for any pair of nodes in $S_q(G)$, using the connection between eigenvalues and eigenvectors of normalized adjacency matrix and hitting time and resistance distance. Moreover, we obtained the Kemeny constant, Kirchhoff index, multiplicative degree-Kirchhoff index, and additive degree-Kirchhoff index for $S_q(G)$. Finally, we derived explicit formulas for some interesting quantities of iterated $q$-subdivisions for any graph $G$, using which we obtained closed-form expressions for those corresponding quantities of the scale-free fractal hierarchical lattices.

It deserves to mention that our computation method and process also apply to other graph operations, such as $q$-triangulation. For a  graph $G$, its $q$-triangulation is a obtained from $G$:  For each edge $e$ in $G$ we create $q$ new nodes,  and  connect them to both end nodes of  $e$. The  $q$-triangulation is a generalization of traditional triangulation operation~\cite{XiZhCo16b}, which has been used to generate scale-free small-world networks~\cite{ZhRoZh07}.

\section*{Acknowledgements}

This work was supported by the National Natural Science Foundation of China under Grant No. 11275049.

%\appendix

%\section{Proof of  Lemma \ref{afore1}}\label{app-1}

%\section{Proof of  Lemma \ref{Afore2}}\label{app-2}

\section*{References}

\bibliographystyle{model1-num-names}
\bibliography{Subdivision}

\begin{thebibliography}{53}
\expandafter\ifx\csname natexlab\endcsname\relax\def\natexlab#1{#1}\fi
\providecommand{\bibinfo}[2]{#2}
\ifx\xfnm\relax \def\xfnm[#1]{\unskip,\space#1}\fi
%Type = Article
\bibitem[{Barab{\'a}si and Albert(1999)}]{BaAl99}
\bibinfo{author}{A.-L. Barab{\'a}si}, \bibinfo{author}{R.~Albert},
\newblock \bibinfo{title}{Emergence of scaling in random networks},
\newblock \bibinfo{journal}{Science} \bibinfo{volume}{286}
  (\bibinfo{year}{1999}) \bibinfo{pages}{509--512}.
%Type = Article
\bibitem[{Watts and Strogatz(1998)}]{WaSt98}
\bibinfo{author}{D.~J. Watts}, \bibinfo{author}{S.~H. Strogatz},
\newblock \bibinfo{title}{Collective dynamics of `small-world' networks},
\newblock \bibinfo{journal}{Nature} \bibinfo{volume}{393}
  (\bibinfo{year}{1998}) \bibinfo{pages}{440--442}.
%Type = Article
\bibitem[{Song et~al.(2005)Song, Havlin, and Makse}]{SoHaMa05}
\bibinfo{author}{C.~Song}, \bibinfo{author}{S.~Havlin}, \bibinfo{author}{H.~A.
  Makse},
\newblock \bibinfo{title}{Self-similarity of complex networks},
\newblock \bibinfo{journal}{Nature} \bibinfo{volume}{433}
  (\bibinfo{year}{2005}) \bibinfo{pages}{392--395}.
%Type = Article
\bibitem[{Newman(2003)}]{Ne03}
\bibinfo{author}{M.~E. Newman},
\newblock \bibinfo{title}{The structure and function of complex networks},
\newblock \bibinfo{journal}{SIAM Rev.} \bibinfo{volume}{45}
  (\bibinfo{year}{2003}) \bibinfo{pages}{167--256}.
%Type = Article
\bibitem[{Girvan and Newman(2002)}]{GiNe02}
\bibinfo{author}{M.~Girvan}, \bibinfo{author}{M.~E. Newman},
\newblock \bibinfo{title}{Community structure in social and biological
  networks},
\newblock \bibinfo{journal}{Proc. Natl. Acad. Sci. U.S.A.} \bibinfo{volume}{99}
  (\bibinfo{year}{2002}) \bibinfo{pages}{7821--7826}.
%Type = Article
\bibitem[{Milo et~al.(2002)Milo, Shen-Orr, Itzkovitz, Kashtan, Chklovskii, and
  Alon}]{MiShItKaKhAl02}
\bibinfo{author}{R.~Milo}, \bibinfo{author}{S.~Shen-Orr},
  \bibinfo{author}{S.~Itzkovitz}, \bibinfo{author}{N.~Kashtan},
  \bibinfo{author}{D.~Chklovskii}, \bibinfo{author}{U.~Alon},
\newblock \bibinfo{title}{Network motifs: {S}imple building blocks of complex
  networks},
\newblock \bibinfo{journal}{Science} \bibinfo{volume}{298}
  (\bibinfo{year}{2002}) \bibinfo{pages}{824--827}.
%Type = Inproceedings
\bibitem[{Tsourakakis(2015)}]{Ts15}
\bibinfo{author}{C.~Tsourakakis},
\newblock \bibinfo{title}{The $k$-clique densest subgraph problem},
\newblock in: \bibinfo{booktitle}{{Proceedings of the 24th International
  Conference on World Wide Web}}, \bibinfo{organization}{ACM}, pp.
  \bibinfo{pages}{1122--1132}.
%Type = Article
\bibitem[{Dorogovtsev et~al.(2002)Dorogovtsev, Goltsev, and Mendes}]{DoGoMe02}
\bibinfo{author}{S.~N. Dorogovtsev}, \bibinfo{author}{A.~V. Goltsev},
  \bibinfo{author}{J.~F.~F. Mendes},
\newblock \bibinfo{title}{Pseudofractal scale-free web},
\newblock \bibinfo{journal}{Phys. Rev. E} \bibinfo{volume}{65}
  (\bibinfo{year}{2002}) \bibinfo{pages}{066122}.
%Type = Article
\bibitem[{Zhang and Comellas(2011)}]{ZhCo11}
\bibinfo{author}{Z.~Zhang}, \bibinfo{author}{F.~Comellas},
\newblock \bibinfo{title}{Farey graphs as models for complex networks},
\newblock \bibinfo{journal}{Theor. Comput. Sci.} \bibinfo{volume}{412}
  (\bibinfo{year}{2011}) \bibinfo{pages}{865--875}.
%Type = Article
\bibitem[{Andrade~Jr et~al.(2005)Andrade~Jr, Herrmann, Andrade, and
  Da~Silva}]{AnHeAnDa05}
\bibinfo{author}{J.~S. Andrade~Jr}, \bibinfo{author}{H.~J. Herrmann},
  \bibinfo{author}{R.~F. Andrade}, \bibinfo{author}{L.~R. Da~Silva},
\newblock \bibinfo{title}{{Apollonian networks: Simultaneously scale-free,
  small world, Euclidean, space filling, and with matching graphs}},
\newblock \bibinfo{journal}{Phys. Rev. Lett.} \bibinfo{volume}{94}
  (\bibinfo{year}{2005}) \bibinfo{pages}{018702}.
%Type = Article
\bibitem[{Doye and Massen(2005)}]{DoMa05}
\bibinfo{author}{J.~P. Doye}, \bibinfo{author}{C.~P. Massen},
\newblock \bibinfo{title}{Self-similar disk packings as model spatial
  scale-free networks},
\newblock \bibinfo{journal}{Phys. Rev. E} \bibinfo{volume}{71}
  (\bibinfo{year}{2005}) \bibinfo{pages}{016128}.
%Type = Article
\bibitem[{Jin et~al.(2017)Jin, Li, and Zhang}]{JiLiZh17}
\bibinfo{author}{Y.~Jin}, \bibinfo{author}{H.~Li}, \bibinfo{author}{Z.~Zhang},
\newblock \bibinfo{title}{{Maximum matchings and minimum dominating sets in
  Apollonian networks and extended Tower of Hanoi graphs}},
\newblock \bibinfo{journal}{Theoret. Comput. Sci.} \bibinfo{volume}{703}
  (\bibinfo{year}{2017}) \bibinfo{pages}{37--54}.
%Type = Article
\bibitem[{Weichsel(1962)}]{We62}
\bibinfo{author}{P.~M. Weichsel},
\newblock \bibinfo{title}{The {K}ronecker product of graphs},
\newblock \bibinfo{journal}{Proc. Am. Math. Soc.} \bibinfo{volume}{13}
  (\bibinfo{year}{1962}) \bibinfo{pages}{47--52}.
%Type = Inproceedings
\bibitem[{Leskovec and Faloutsos(2007)}]{LeFa07}
\bibinfo{author}{J.~Leskovec}, \bibinfo{author}{C.~Faloutsos},
\newblock \bibinfo{title}{Scalable modeling of real graphs using {K}ronecker
  multiplication},
\newblock in: \bibinfo{booktitle}{Proceedings of the 24th International
  Conference on Machine Learning}, \bibinfo{publisher}{ACM},
  \bibinfo{address}{New York, NY, USA}, \bibinfo{year}{2007}, pp.
  \bibinfo{pages}{497--504}.
%Type = Article
\bibitem[{Leskovec et~al.(2010)Leskovec, Chakrabarti, Kleinberg, Faloutsos, and
  Ghahramani}]{LeChKlFaGh10}
\bibinfo{author}{J.~Leskovec}, \bibinfo{author}{D.~Chakrabarti},
  \bibinfo{author}{J.~Kleinberg}, \bibinfo{author}{C.~Faloutsos},
  \bibinfo{author}{Z.~Ghahramani},
\newblock \bibinfo{title}{Kronecker graphs: An approach to modeling networks},
\newblock \bibinfo{journal}{J. Mach. Learn. Res.} \bibinfo{volume}{11}
  (\bibinfo{year}{2010}) \bibinfo{pages}{985--1042}.
%Type = Article
\bibitem[{Barriere et~al.(2009)Barriere, Comellas, Dalf{\'o}, and
  Fiol}]{BaCoDaFi09}
\bibinfo{author}{L.~Barriere}, \bibinfo{author}{F.~Comellas},
  \bibinfo{author}{C.~Dalf{\'o}}, \bibinfo{author}{M.~A. Fiol},
\newblock \bibinfo{title}{The hierarchical product of graphs},
\newblock \bibinfo{journal}{Discrete Appl. Math.} \bibinfo{volume}{157}
  (\bibinfo{year}{2009}) \bibinfo{pages}{36--48}.
%Type = Article
\bibitem[{Barri{\`e}re et~al.(2009)Barri{\`e}re, Dalf{\'o}, Fiol, and
  Mitjana}]{BaDaFiMi09}
\bibinfo{author}{L.~Barri{\`e}re}, \bibinfo{author}{C.~Dalf{\'o}},
  \bibinfo{author}{M.~A. Fiol}, \bibinfo{author}{M.~Mitjana},
\newblock \bibinfo{title}{The generalized hierarchical product of graphs},
\newblock \bibinfo{journal}{Discrete Math.} \bibinfo{volume}{309}
  (\bibinfo{year}{2009}) \bibinfo{pages}{3871--3881}.
%Type = Article
\bibitem[{Barriere et~al.(2016)Barriere, Comellas, Dalfo, and
  Fiol}]{BaCoDaFi16}
\bibinfo{author}{L.~Barriere}, \bibinfo{author}{F.~Comellas},
  \bibinfo{author}{C.~Dalfo}, \bibinfo{author}{M.~Fiol},
\newblock \bibinfo{title}{Deterministic hierarchical networks},
\newblock \bibinfo{journal}{J. Phys. A: Math. Theoret.} \bibinfo{volume}{49}
  (\bibinfo{year}{2016}) \bibinfo{pages}{225202}.
%Type = Article
\bibitem[{Lv et~al.(2015)Lv, Yi, and Zhang}]{LvYiZh15}
\bibinfo{author}{Q.~Lv}, \bibinfo{author}{Y.~Yi}, \bibinfo{author}{Z.~Zhang},
\newblock \bibinfo{title}{Corona graphs as a model of small-world networks},
\newblock \bibinfo{journal}{J. Stat. Mech.} \bibinfo{volume}{2015}
  (\bibinfo{year}{2015}) \bibinfo{pages}{P11024}.
%Type = Article
\bibitem[{Sharma et~al.(2017)Sharma, Adhikari, and Mishra}]{ShAdMi17}
\bibinfo{author}{R.~Sharma}, \bibinfo{author}{B.~Adhikari},
  \bibinfo{author}{A.~Mishra},
\newblock \bibinfo{title}{Structural and spectral properties of corona graphs},
\newblock \bibinfo{journal}{Discrete Appl. Math.} \bibinfo{volume}{228}
  (\bibinfo{year}{2017}) \bibinfo{pages}{14–31}.
%Type = Article
\bibitem[{Qi et~al.(2018)Qi, Li, and Zhang}]{QiLiZh18}
\bibinfo{author}{Y.~Qi}, \bibinfo{author}{H.~Li}, \bibinfo{author}{Z.~Zhang},
\newblock \bibinfo{title}{Extended corona product as an exactly tractable model
  for weighted heterogeneous networks},
\newblock \bibinfo{journal}{Comput. J.} \bibinfo{volume}{61}
  (\bibinfo{year}{2018}) \bibinfo{pages}{***--*** (in press)}.
%Type = Article
\bibitem[{Wood(2005)}]{Wo05}
\bibinfo{author}{D.~R. Wood},
\newblock \bibinfo{title}{Acyclic, star and oriented colourings of graph
  subdivisions},
\newblock \bibinfo{journal}{Discrete Math. Theoret. Comput. Sci.}
  \bibinfo{volume}{7} (\bibinfo{year}{2005}) \bibinfo{pages}{37--50}.
%Type = Article
\bibitem[{Hu et~al.(2015)Hu, Yan, and Qiu}]{HuYaQi15}
\bibinfo{author}{M.~Hu}, \bibinfo{author}{W.~Yan}, \bibinfo{author}{W.~Qiu},
\newblock \bibinfo{title}{Maximal energy of subdivisions of graphs with a fixed
  chromatic number},
\newblock \bibinfo{journal}{Bull. Malays. Math. Sci. Soc.} \bibinfo{volume}{38}
  (\bibinfo{year}{2015}) \bibinfo{pages}{1349--1359}.
%Type = Article
\bibitem[{Carmona et~al.(2016)Carmona, Mitjana, and Mons{\'o}}]{CaMiMo16}
\bibinfo{author}{A.~Carmona}, \bibinfo{author}{M.~Mitjana},
  \bibinfo{author}{E.~Mons{\'o}},
\newblock \bibinfo{title}{The group inverse of subdivision networks},
\newblock \bibinfo{journal}{Electron. Notes Discrete Math.}
  \bibinfo{volume}{54} (\bibinfo{year}{2016}) \bibinfo{pages}{295--300}.
%Type = Book
\bibitem[{West(2001)}]{We01}
\bibinfo{author}{D.~B. West}, \bibinfo{title}{Introduction to Graph Theory},
  volume \bibinfo{volume}{2nd ed.}, \bibinfo{publisher}{Prentice hall, Upper
  Saddle River}, \bibinfo{year}{2001}.
%Type = Article
\bibitem[{Fiedorowicz and Ha{\l}uszczak(2012)}]{FiHa12}
\bibinfo{author}{A.~Fiedorowicz}, \bibinfo{author}{M.~Ha{\l}uszczak},
\newblock \bibinfo{title}{Acyclic chromatic indices of fully subdivided
  graphs},
\newblock \bibinfo{journal}{Inform. Process. Lett.} \bibinfo{volume}{112}
  (\bibinfo{year}{2012}) \bibinfo{pages}{557--561}.
%Type = Article
\bibitem[{Yang(1988)}]{Ya88}
\bibinfo{author}{Z.~R. Yang},
\newblock \bibinfo{title}{Family of diamond-type hierarchical lattices},
\newblock \bibinfo{journal}{Phys. Rev. B} \bibinfo{volume}{38}
  (\bibinfo{year}{1988}) \bibinfo{pages}{728}.
%Type = Article
\bibitem[{Zhang et~al.(2007)Zhang, Zhou, and Zou}]{ZhZhZo07}
\bibinfo{author}{Z.-Z. Zhang}, \bibinfo{author}{S.-G. Zhou},
  \bibinfo{author}{T.~Zou},
\newblock \bibinfo{title}{Self-similarity, small-world, scale-free scaling,
  disassortativity, and robustness in hierarchical lattices},
\newblock \bibinfo{journal}{Eur. Phys. J. B} \bibinfo{volume}{56}
  (\bibinfo{year}{2007}) \bibinfo{pages}{259--271}.
%Type = Article
\bibitem[{Zhang et~al.(2011)Zhang, Yang, and Gao}]{ZhYaGa11}
\bibinfo{author}{Z.~Zhang}, \bibinfo{author}{Y.~Yang},
  \bibinfo{author}{S.~Gao},
\newblock \bibinfo{title}{Role of fractal dimension in random walks on
  scale-free networks},
\newblock \bibinfo{journal}{Eur. Phys. J. B} \bibinfo{volume}{84}
  (\bibinfo{year}{2011}) \bibinfo{pages}{331--338}.
%Type = Article
\bibitem[{Zhang et~al.(2012)Zhang, Sheng, Hu, and Chen}]{ZhShHuCh12}
\bibinfo{author}{Z.~Zhang}, \bibinfo{author}{Y.~Sheng},
  \bibinfo{author}{Z.~Hu}, \bibinfo{author}{G.~Chen},
\newblock \bibinfo{title}{Optimal and suboptimal networks for efficient
  navigation measured by mean-first passage time of random walks},
\newblock \bibinfo{journal}{Chaos} \bibinfo{volume}{22} (\bibinfo{year}{2012})
  \bibinfo{pages}{043129}.
%Type = Article
\bibitem[{Li and Zhang(2017)}]{LiZh17}
\bibinfo{author}{H.~Li}, \bibinfo{author}{Z.~Zhang},
\newblock \bibinfo{title}{Maximum matchings in scale-free networks with
  identical degree distribution},
\newblock \bibinfo{journal}{Theoret. Comput. Sci.} \bibinfo{volume}{675}
  (\bibinfo{year}{2017}) \bibinfo{pages}{64--81}.
%Type = Book
\bibitem[{Cvetkovi{\'c} et~al.(1980)Cvetkovi{\'c}, Doob, and Sachs}]{CvDoSa80}
\bibinfo{author}{D.~M. Cvetkovi{\'c}}, \bibinfo{author}{M.~Doob},
  \bibinfo{author}{H.~Sachs}, \bibinfo{title}{Spectra of graphs: theory and
  application}, volume~\bibinfo{volume}{87}, \bibinfo{publisher}{New York, NY,
  USA, Academic Press}, \bibinfo{year}{1980}.
%Type = Book
\bibitem[{Chung(1997)}]{Ch97}
\bibinfo{author}{F.~R.~K. Chung}, \bibinfo{title}{Spectral graph theory},
  \bibinfo{number}{92}, \bibinfo{publisher}{American Mathematical Society},
  \bibinfo{year}{1997}.
%Type = Book
\bibitem[{Redner(2001)}]{Re01}
\bibinfo{author}{S.~Redner}, \bibinfo{title}{A guide to first-passage
  processes}, \bibinfo{publisher}{Cambridge University Press},
  \bibinfo{address}{Cambridge, UK}, \bibinfo{year}{2001}.
%Type = Article
\bibitem[{Hunter(2014)}]{Hu14}
\bibinfo{author}{J.~J. Hunter},
\newblock \bibinfo{title}{{The role of Kemeny's constant in properties of
  Markov chains}},
\newblock \bibinfo{journal}{Commun. Stat. --- Theor. Methods}
  \bibinfo{volume}{43} (\bibinfo{year}{2014}) \bibinfo{pages}{1309--1321}.
%Type = Article
\bibitem[{Levene and Loizou(2002)}]{LeLo02}
\bibinfo{author}{M.~Levene}, \bibinfo{author}{G.~Loizou},
\newblock \bibinfo{title}{Kemeny's constant and the random surfer},
\newblock \bibinfo{journal}{Am. Math. Mon.} \bibinfo{volume}{109}
  (\bibinfo{year}{2002}) \bibinfo{pages}{741--745}.
%Type = Article
\bibitem[{Lov{\'a}sz(1993)}]{Lo93}
\bibinfo{author}{L.~Lov{\'a}sz},
\newblock \bibinfo{title}{Random walks on graphs},
\newblock \bibinfo{journal}{Combinatorics, Paul {E}rd\"{o}s is eighty}
  \bibinfo{volume}{2} (\bibinfo{year}{1993}) \bibinfo{pages}{4}.
%Type = Incollection
\bibitem[{Butler(2016)}]{butler2016algebraic}
\bibinfo{author}{S.~Butler},
\newblock \bibinfo{title}{Algebraic aspects of the normalized {L}aplacian},
\newblock in: \bibinfo{booktitle}{Recent Trends in Combinatorics},
  \bibinfo{publisher}{Springer}, \bibinfo{year}{2016}, pp.
  \bibinfo{pages}{295--315}.
%Type = Book
\bibitem[{Doyle and Snell(1984)}]{DoSn84}
\bibinfo{author}{P.~G. Doyle}, \bibinfo{author}{J.~L. Snell},
  \bibinfo{title}{Random Walks and Electric Networks},
  \bibinfo{publisher}{Mathematical Association of America},
  \bibinfo{year}{1984}.
%Type = Article
\bibitem[{Chen and Zhang(2007)}]{ChZh07}
\bibinfo{author}{H.~Chen}, \bibinfo{author}{F.~Zhang},
\newblock \bibinfo{title}{{Resistance distance and the normalized Laplacian
  spectrum}},
\newblock \bibinfo{journal}{Discrete Appl. Math.} \bibinfo{volume}{155}
  (\bibinfo{year}{2007}) \bibinfo{pages}{654--661}.
%Type = Article
\bibitem[{Foster(1949)}]{Fo49}
\bibinfo{author}{R.~M. Foster},
\newblock \bibinfo{title}{The average impedance of an electrical network},
\newblock \bibinfo{journal}{Contributions to Applied Mechanics (Reissner
  Anniversary Volume)}  (\bibinfo{year}{1949}) \bibinfo{pages}{333--340}.
%Type = Inproceedings
\bibitem[{Chandra et~al.(1989)Chandra, Raghavan, Ruzzo, and
  Smolensky}]{ChRaRuSm89}
\bibinfo{author}{A.~K. Chandra}, \bibinfo{author}{P.~Raghavan},
  \bibinfo{author}{W.~L. Ruzzo}, \bibinfo{author}{R.~Smolensky},
\newblock \bibinfo{title}{The electrical resistance of a graph captures its
  commute and cover times},
\newblock in: \bibinfo{booktitle}{Proc. 21st Ann. ACM Symp. Theory Comput.},
  \bibinfo{organization}{ACM}, pp. \bibinfo{pages}{574--586}.
%Type = Article
\bibitem[{Ghosh et~al.(2008)Ghosh, Boyd, and Saberi}]{GhBoSa08}
\bibinfo{author}{A.~Ghosh}, \bibinfo{author}{S.~Boyd},
  \bibinfo{author}{A.~Saberi},
\newblock \bibinfo{title}{Minimizing effective resistance of a graph},
\newblock \bibinfo{journal}{SIAM Rev.} \bibinfo{volume}{50}
  (\bibinfo{year}{2008}) \bibinfo{pages}{37--66}.
%Type = Article
\bibitem[{Klein and Randi{\'c}(1993)}]{KlRa93}
\bibinfo{author}{D.~J. Klein}, \bibinfo{author}{M.~Randi{\'c}},
\newblock \bibinfo{title}{Resistance distance},
\newblock \bibinfo{journal}{J. Math. Chem.} \bibinfo{volume}{12}
  (\bibinfo{year}{1993}) \bibinfo{pages}{81--95}.
%Type = Article
\bibitem[{Tizghadam and Leon-Garcia(2010)}]{TiLe10}
\bibinfo{author}{A.~Tizghadam}, \bibinfo{author}{A.~Leon-Garcia},
\newblock \bibinfo{title}{Autonomic traffic engineering for network
  robustness},
\newblock \bibinfo{journal}{IEEE J. Sel. Areas Commun.} \bibinfo{volume}{28}
  (\bibinfo{year}{2010}).
%Type = Article
\bibitem[{Patterson and Bamieh(2014)}]{PaBa14}
\bibinfo{author}{S.~Patterson}, \bibinfo{author}{B.~Bamieh},
\newblock \bibinfo{title}{Consensus and coherence in fractal networks},
\newblock \bibinfo{journal}{IEEE Trans. Control Netw. Syst.}
  \bibinfo{volume}{1} (\bibinfo{year}{2014}) \bibinfo{pages}{338--348}.
%Type = Inproceedings
\bibitem[{Li and Zhang(2018)}]{LiZh18}
\bibinfo{author}{H.~Li}, \bibinfo{author}{Z.~Zhang},
\newblock \bibinfo{title}{Kirchhoff index as a measure of edge centrality in
  weighted networks: {N}early linear time algorithms},
\newblock in: \bibinfo{booktitle}{Proceedings of the 29th Annual ACM-SIAM
  Symposium on Discrete Algorithms}, pp. \bibinfo{pages}{2377--2396}.
%Type = Article
\bibitem[{Gutman et~al.(2012)Gutman, Feng, and Yu}]{GuFeYu12}
\bibinfo{author}{I.~Gutman}, \bibinfo{author}{L.~Feng},
  \bibinfo{author}{G.~Yu},
\newblock \bibinfo{title}{Degree resistance distance of unicyclic graphs},
\newblock \bibinfo{journal}{Trans. Combin.} \bibinfo{volume}{1}
  (\bibinfo{year}{2012}) \bibinfo{pages}{27--40}.
%Type = Article
\bibitem[{Xie et~al.(2016)Xie, Zhang, and Comellas}]{XiZhC16a}
\bibinfo{author}{P.~Xie}, \bibinfo{author}{Z.~Zhang},
  \bibinfo{author}{F.~Comellas},
\newblock \bibinfo{title}{{The normalized Laplacian spectrum of subdivisions of
  a graph}},
\newblock \bibinfo{journal}{Appl. Math. Comput.} \bibinfo{volume}{286}
  (\bibinfo{year}{2016}) \bibinfo{pages}{250--256}.
%Type = Article
\bibitem[{Yang(2014)}]{Ya14}
\bibinfo{author}{Y.~Yang},
\newblock \bibinfo{title}{{The Kirchhoff index of subdivisions of graphs}},
\newblock \bibinfo{journal}{Discrete Appl. Math.} \bibinfo{volume}{171}
  (\bibinfo{year}{2014}) \bibinfo{pages}{153--157}.
%Type = Article
\bibitem[{Yang and Klein(2015)}]{YaKl15}
\bibinfo{author}{Y.~Yang}, \bibinfo{author}{D.~J. Klein},
\newblock \bibinfo{title}{Resistance distance-based graph invariants of
  subdivisions and triangulations of graphs},
\newblock \bibinfo{journal}{Discrete Appl. Math.} \bibinfo{volume}{181}
  (\bibinfo{year}{2015}) \bibinfo{pages}{260--274}.
%Type = Article
\bibitem[{Xie et~al.(2016)Xie, Zhang, and Comellas}]{XiZhCo16b}
\bibinfo{author}{P.~Xie}, \bibinfo{author}{Z.~Zhang},
  \bibinfo{author}{F.~Comellas},
\newblock \bibinfo{title}{{On the spectrum of the normalized Laplacian of
  iterated triangulations of graphs}},
\newblock \bibinfo{journal}{Appl. Math. Comput.} \bibinfo{volume}{273}
  (\bibinfo{year}{2016}) \bibinfo{pages}{1123--1129}.
%Type = Article
\bibitem[{Zhang et~al.(2007)Zhang, Rong, and Zhou}]{ZhRoZh07}
\bibinfo{author}{Z.~Zhang}, \bibinfo{author}{L.~Rong},
  \bibinfo{author}{S.~Zhou},
\newblock \bibinfo{title}{A general geometric growth model for pseudofractal
  scale-free web},
\newblock \bibinfo{journal}{Physica A} \bibinfo{volume}{377}
  (\bibinfo{year}{2007}) \bibinfo{pages}{329--339}.

\end{thebibliography}

\end{document}